\documentclass[12pt]{article}

\usepackage{amsmath,amsfonts,yhmath,geometry,hyperref}
\usepackage{amsthm}
\usepackage{amssymb}
\usepackage[latin9]{inputenc}
\usepackage[nottoc,notlof,notlot]{tocbibind}
\usepackage{graphicx,float,psfrag}
\usepackage{mathtools}
\usepackage{floatflt}
\usepackage{ulem}
\usepackage{enumitem}

\usepackage{tikz,pgfplots}
\usepackage{mathrsfs}
\usetikzlibrary{arrows}
\usetikzlibrary[patterns]

\definecolor{qqqqff}{rgb}{0.,0.,1.}
\definecolor{cqcqcq}{rgb}{0.7529411764705882,0.7529411764705882,0.7529411764705882}
\definecolor{ttqqqq}{rgb}{0.2,0.,0.}
\definecolor{qqqqff}{rgb}{0.,0.,1.}
\definecolor{xdxdff}{rgb}{0.49019607843137253,0.49019607843137253,1.}
\definecolor{zzttqq}{rgb}{0.6,0.2,0.}
\definecolor{cqcqcq}{rgb}{0.7529411764705882,0.7529411764705882,0.7529411764705882}

\definecolor{yqyqyq}{rgb}{0.5019607843137255,0.5019607843137255,0.5019607843137255}
\definecolor{uuuuuu}{rgb}{0.26666666666666666,0.26666666666666666,0.26666666666666666}
\definecolor{xdxdff}{rgb}{0.49019607843137253,0.49019607843137253,1.}
\definecolor{qqqqff}{rgb}{0.,0.,1.}

 \font\ncsc=cmcsc10
 \font\ntt=cmtt12

\usepackage{fancybox}

\newcommand{\PP}{\mathbb{P}}

\newcommand{\ZZ}{\mathbb{Z}}
\newcommand{\RR}{\mathbb{R}}
\newcommand{\CC}{\mathbb{C}}

\newcommand{\dd}{ \text{d} }

\newcommand{\partrop}{{\partial,\text{trop}}}

\newcommand{\der}[1]{\frac{\dd #1}{\dd t}}

\newtheorem{theo}{Theorem}[section]
\newtheorem{prop}[theo]{Proposition}

\theoremstyle{definition}
\newtheorem{defi}[theo]{Definition}

\theoremstyle{remark}
\newtheorem{remark}[theo]{Remark}
\newenvironment{rem}[1]{
    \begin{remark}#1}{
    \xqed{\blacklozenge}\end{remark}
}

\theoremstyle{remark}
\newtheorem{example}[theo]{Example}
\newenvironment{expl}[1]{
    \begin{example}#1}{
    \xqed{\lozenge}\end{example}
}

\newcommand{\xqed}[1]{
    \leavevmode\unskip\penalty9999 \hbox{}\nobreak\hfill
    \quad\hbox{\ensuremath{#1}}}

\usepackage[off]{auto-pst-pdf}

\begin{document}
 
 
\title{A Caporaso-Harris type Formula for relative refined invariants}
\author{Thomas Blomme}
\maketitle

\begin{abstract}
    In \cite{mikhalkin2017quantum}, G. Mikhalkin introduced a refined count for real rational curves  in a toric surface which pass through some points on the toric boundary of the surface. The refinement is provided by the value of a so-called quantum index. Moreover, he proved that the result of this refined count does not depend on the choice of the points. The correspondence theorem allows one to compute these invariants using the tropical geometry approach and the refined Block-G\"ottsche multiplicities. In this paper we give a recursive formula for these invariants, that leads to an algorithm to compute them.
\end{abstract}

\tableofcontents

\section{Introduction}


The paper deals with enumerative problems involving rational curves in toric surfaces. Let $N$ be a $2$-dimensional lattice with basis $(e_1,e_2)$, and let $M$ be the dual lattice, with the dual basis $(e_1^*,e_2^*)$. The lattices $M$ and $N$ are called respectively, the lattices of characters and co-characters. Let $\Delta=(n_j)\subset N$ be a multiset of lattice vectors, whose total sum is zero. We denote by $m$ the cardinal of $\Delta$. Let $P_\Delta\subset M$ be the convex lattice polygon in $M$, defined up to translation, whose normal vectors oriented outside $P_\Delta$ are the vectors of $\Delta$, counted with multiplicity. This means that the lattice length of a side $E$ of $P_\Delta$ is precisely the sum of the lattice lengths of the vectors of $\Delta$ which are normal to the side $E$. The number of lattice points on the boundary of $P_\Delta$ is thus equal to the sum of the lattice lengths of the vectors of $\Delta$. The vectors of $\Delta$ define a fan $\Sigma_\Delta$ in $N_\RR$, only depending on $P_\Delta$. We denote by $\CC\Delta$ the associated toric surface, whose dense complex torus is $\text{Hom}(M,\CC^*)\simeq N\otimes\CC^*$. The toric divisors  of $\CC\Delta$ are in bijection with the sides of the polygon $P_\Delta$. The complex conjugation on $\CC$ defines an involution on $\text{Hom}(M,\CC^*)$, which extends to $\CC\Delta$. This additional structure makes it into a real surface. Its fixed locus, also called the real locus, is denoted by $\RR\Delta$.

\begin{expl}
\begin{itemize}[label=-]
\item Let $\Delta_d=\{-e_1^d,-e_2^d,(e_1+e_2)^d\}$, where the $d$ exponent means that the vector is present $d$ times in the multiset. The associated polygon is $P_d=\text{Conv}( 0,d\cdot e_1^*,d\cdot e_2^* )$, the standard triangle of size $d$. The associated toric surface is the projective plane $\CC P^2$, and the real locus is the real projective plane $\RR P^2$.
\item If we set $\square_{a,b}=\{-e_1^b,e_1^b,-e_2^a,e_2^b\}$, the polygon $P_{a,b}$ is the rectangle $\text{Conv}( 0,a\cdot e_1^*,b\cdot e_2^*,a\cdot e_1^*+b\cdot e_2^* )$, and the associated toric surface is $\CC P^1\times\CC P^1$, whose real locus is $\RR P^1\times\RR P^1$.
\end{itemize}
\end{expl}
A rational curve $\CC P^1\rightarrow\CC \Delta$ of degree $\Delta$ is a curve that admits a parametrization
$$\varphi :t\in\CC \dashrightarrow \chi\prod_{i=1}^m(t-\alpha_i)^{n_i}\in\text{Hom}(M,\CC^*),$$
where the numbers $\alpha_i$ are some complex points inside $\CC$, and $\chi:M\rightarrow\CC^*$ is a complex co-character, \textit{i.e.} an element of $N\otimes\CC^*$. The vectors $n_j$ can also be recovered as the functions $m\mapsto\text{val}_{\alpha_j}(\chi^m)$, where $\chi^m$ denotes the monomial function associated to the character $m\in M$, and $\text{val}_{\alpha_j}$ is the order of vanishing at $\alpha_j$. Such a parametrization is unique up to the automorphisms of $\CC P^1$, therefore the space of rational curves $\CC P^1\rightarrow\CC\Delta$ is of dimension $m-1$. A rational curve is real if it admits a parametrization invariant by conjugation. Equivalently, $\chi$ should have real values, and the numbers $\alpha_i$ are either real, or come in pairs of conjugated points with the same exponent vector $n_j$.\\

We choose a generic configuration $\mathcal{P}$ of $m-1$ points inside $\CC\Delta$, and look for rational curves passing through this configuration. We have a finite number of complex solutions. The cardinal $|\mathcal{S}^\CC(\mathcal{P})|$ of the set $\mathcal{S}^\CC(\mathcal{P})$ of solutions is independent of the choice of the point configuration $\mathcal{P}$, and the value of this cardinal is denoted by $N_\Delta$. We now consider a configuration of points $\mathcal{P}$ which is invariant by the complex conjugation, \textit{i.e.} it consists of real points and pairs of conjugated points. Such a configuration is called a \textit{real configuration of points}. If all the points are real, it is a configuration of real points, otherwise we call it a \textit{non-totally real configuration}. Contrarily to the complex case, the number $|\mathcal{S}^\RR(\mathcal{P})|$ of real curves passing through the configuration $\mathcal{P}$ might depend on the chosen configuration. However, in \cite{welschinger2005invariants} J-Y. Welschinger showed that if the toric surface $\CC\Delta$ is a del Pezzo surface and the curves are counted with an appropriate sign, the number of solutions in $\mathcal{S}^\RR(\mathcal{P})$ only depends on the number of pairs of conjugated points in the configuration. This invariant is called Welschinger invariant, often denoted by $W_{\Delta,s}$, where $s$ is the number of pairs of conjugated points.\\

The values of $N_\Delta$ were already known when Welschinger proved his invariants to be invariants. And roughly at the same time, Mikhalkin \cite{mikhalkin2005enumerative} proved a correspondence theorem that provided a way of computing both invariants $N_\Delta$ and Welschinger invariants $W_{\Delta,0}$ in the case of a configuration of real points, using the tropical geometry approach. Later E. Shustin \cite{shustin2002patchworking} also computed the invariants $W_{\Delta,0}$, and managed in \cite{shustin2004tropical} to make a tropical calculation of the $W_{\Delta,s}$ for any $s$, \textit{i.e.} not only in the case of a totally real configuration.\\

To compute the values of $N_\Delta$ and $W_{\Delta,0}$, Mikhalkin counts tropical curves solution to a similar enumerative problem with two specific choices of multiplicity. Following his computation, F. Block and L. G\"ottsche \cite{block2016refined} proposed a way of combining these integer multiplicities by refining them into a Laurent polynomial multiplicity. This polynomial multiplicity evaluated at $\pm 1$ gives back the multiplicities used to compute the invariants $N_\Delta$ and $W_{\Delta,0}$. Moreover, the counting of tropical curves of fixed degree passing through a generic configuration of points using refined multiplicity was proved in \cite{itenberg2013block} to give a tropical invariant. This new choice of multiplicity seems to appear in more and more situations, while its meaning in classical geometry remains quite mysterious. Conjecturally, the refined invariant coincides with the refinement of Severi degrees by the $\chi_{-y}$-genera proposed by L. G\"ottsche and V. Shende in \cite{gottsche2014refined}. This invariant bears also similarities with some Donaldson-Thomas wall-crossing invariants considered by M. Kontsevich and Y. Soibelman \cite{kontsevich2008stability}.\\

In \cite{mikhalkin2017quantum} Mikhalkin introduced a quantum index for real type I curves having real or purely imaginary intersection points with the toric boundary. Not yet going into further details, this is the case of real rational curves with real intersection points with the toric divisors. He then proved that a signed count of real rational curves passing through a symmetric real configuration of points according to the value of their quantum index only depends on the number of pairs of complex conjugated points in the configuration. The details of the latter sentence are explained after Theorem \ref{quantumindex}. He finally proved that in the case of a configuration of real points (\textit{i.e.} no pairs of complex conjugated points), this new invariant could be computed via tropical geometry using the same Block-G\"ottsche refined multiplicity. This provides another interpretation of these mysterious refined tropical invariants.\\

Let $\varphi:\CC C\rightarrow\CC\Delta$ be a parametrized real curve of degree $\Delta$, with $\CC C$ a smooth Riemann surface. It is of type I if $\CC C\backslash\RR C$ is disconnected, and therefore consists of two conjugated components. Let $S$ be one of these components. It induces an orientation of its boundary $\RR C$, also called a complex orientation of $\RR C$. The choice of the other connected component provides the reversed orientation. When $S$ is given, we denote by $\overrightarrow{C}$ the curve endowed with the orientation induced by $S$. The map $n\otimes z\mapsto 2 n\otimes z$ from $N\otimes\CC^*$ to itself extends to a map $\text{Sq}$ from $\CC\Delta$ to itself, called the square map. In coordinates it is just the map that squares each coordinate. We say that $\CC C$ has real or purely imaginary intersection points with the toric divisors if the images of the intersection points under the square map are real, \textit{i.e.} belong to $\RR\Delta$. We have the logarithmic map
$$\text{Log}:n\otimes z\in N\otimes\CC^*\longmapsto n\otimes\text{Log}|z|\in N_\RR=N\otimes\RR.$$
In a basis of $N$, it is just the logarithm of the absolute value coordinate by coordinate. Let $\omega$ be a generator of $\Lambda^2 M$, \textit{i.e.} a non-degenerated $2$-form on $N$. It extends to a volume form on $N_\RR$. By pulling back the volume form to $N\otimes\CC^*$, we can define the logarithmic area of $S$ :
$$\mathcal{A}_{\text{Log}}(\overrightarrow{C})=\int_{\varphi(S)}\text{Log}^*\omega.$$
The other connected component of $\CC C\backslash\RR C$ has opposite log-area.

\begin{theo}[Mikhalkin \cite{mikhalkin2017quantum}]
\label{quantumindex}
Let $\CC C\subset\CC\Delta$ be a real type I curve with real or purely imaginary intersection points with the toric divisors, endowed with the choice $S$ of a connected component of $\CC C\backslash\RR C$, inducing a complex orientation of $\RR C$. Then, there exists an half-integer $k$, called the quantum index of the oriented curve $\overrightarrow{C}$, such that
$$\mathcal{A}_{\text{Log}}(\overrightarrow{C})=k\pi^2.$$
\end{theo}

This theorem concerns all real type I curves but will only be of interest to us in the case of rational curves with real intersection points with the toric boundary. Such curves have a well-defined quantum index that can be computed through various methods described in \cite{mikhalkin2017quantum}. Moreover, being equal to the logarithmic area, the quantum index changes its sign if we change the complex orientation of the curve.\\

Now, let $\mathcal{P}_0$ be a configuration of $m-1$ real points taken on the toric boundary of $\RR\Delta$, and such that each irreducible component of the boundary does not contain a number of points greater than the integral length of the corresponding side of the polygon $P_\Delta$, \textit{e.g.} for the projective plane, there are two axes with $d$ points and the last axis only has $d-1$. The Vi\`ete formula, or the Weil reciprocity law, ensures that there exists a last $m$th point on the boundary such that each curve of degree $\Delta$ passing through $\mathcal{P}_0$ also passes through this point. For each point $p\in\CC\Delta$, the action of the torus element given by the point $(-1,-1)$ in any coordinates defines a symmetric point $-p$. If $p$ is inside some toric divisor $D$, $-p$ is also in $D$. Let $\mathcal{P}=\mathcal{P}_0\cup(-\mathcal{P}_0)=\{ \pm p_i\ | \ p_i\in\mathcal{P}_0 \}$ be the symmetric configuration of points obtained from $\mathcal{P}_0$. Let $\mathcal{S}(\mathcal{P})$ be the set of real oriented rational curves of degree $\Delta$ passing through either $p_i$ or $-p_i$ for every $i$. In that case we say that a curve \textit{passes through the symmetric configuration}, although it passes only through half of it. Equivalently, if $\text{Sq}:\CC\Delta\rightarrow\CC\Delta$ denotes the square map, the set $\mathcal{S}(\mathcal{P})$ is the set of real oriented rational curves $C$ of degree $\Delta$ such that $\text{Sq}(\RR C)\supset\text{Sq}(\mathcal{P}_0)$. As the curves of $\mathcal{S}(\mathcal{P})$ are oriented, each curve comes twice, once with each of its possible orientations. All the oriented curves of $\mathcal{S}(\mathcal{P})$ have a well-defined quantum index. For an oriented curve $\overrightarrow{C}\in\mathcal{S}(\mathcal{P})$, let $\varphi:\CC P^1\rightarrow \CC C$ be a parametrization of $C$. We consider the composition of the parametrization $\varphi$ with the logarithm map, restricted to the real locus :
$$\text{Log}|\varphi|:\RR P^1\rightarrow N_\RR.$$
Its image is $\text{Log}(\RR C)\subset N_\RR$. We now consider the logarithmic Gauss map that associates to each point $x\in\RR C$ the tangent direction to $\text{Log}(\RR C)$ at the point $\text{Log}x$. This tangent direction lies in in $\PP^1(N_\RR)\simeq\RR P^1$, oriented by $\omega$. Let $\text{Rot}_\text{Log}(\overrightarrow{C})$ be the degree of this map when the parametrizing $\RR P^1$ is oriented by the choice of the complex orientation of $C$, and $\sigma(\overrightarrow{C})=(-1)^\frac{m-\text{Rot}_\text{Log}(\overrightarrow{C})}{2}=\pm 1$. We then set
$$R_{\Delta,k}(\mathcal{P})=\sum_{C\in\mathcal{S}_k(\mathcal{P})} \sigma(\overrightarrow{C})\in\ZZ,$$
where $\mathcal{S}_k(\mathcal{P})$ denotes the subset of $\mathcal{S}(\mathcal{P})$ of oriented curves having quantum index $k$. Finally, we define
$$R_\Delta(\mathcal{P})=\frac{1}{4}\sum_{k}R_{\Delta,k}(\mathcal{P})q^k\in \ZZ\left[ q^{\pm\frac{1}{2}}\right].$$

\begin{theo}[Mikhalkin\cite{mikhalkin2017quantum}]
The Laurent polynomial $R_\Delta(\mathcal{P})$ does not depend on the configuration of points $\mathcal{P}$ as long as it is generic.
\end{theo}

\begin{rem}
This theorem is, in fact, more general since it does not concern only configurations of real points but also real configuration of points, which may contain pairs of conjugated purely imaginary points. However, the relation to refined tropical enumeration is for now only known in the case of a totally real configuration and that is why we restrict to it.
\end{rem}

\begin{rem}
It is important to the result and for its proof not only to consider curves going through $\mathcal{P}_0$ but also through the symmetric configuration $\mathcal{P}$, otherwise the invariance fails.
\end{rem}

In the tropical world, one can define a similar problem: finding rational tropical curves passing through a fixed configuration of points at infinity. The count of the curves with the Block-G\"ottsche multiplicities does not depend on the chosen configuration of point and its value is denoted by $N_\Delta^\partrop$.
In the case of a configuration of real points, using the correspondence theorem, Mikhalkin proved that the invariant $R_\Delta$ coincides up to a normalization with $N_\Delta^{\partial\text{trop}}$.

\begin{theo}[Mikhalkin\cite{mikhalkin2017quantum}]
One has
$$R_\Delta=(q^{1/2}-q^{-1/2})^{m-2}N_\Delta^{\partial,\text{trop}}.$$
\end{theo}

\begin{rem}
The normalization $(q^{1/2}-q^{-1/2})^{m-2}$ amounts to clear the denominators of the Block-G\"ottsche multiplicities.
\end{rem}

The results of Mikhalkin reduce the computation of the invariants $R_\Delta$ to a tropical one for the invariants $N_\Delta^{\partial,\text{trop}}$. In this paper we prove a recursive formula for the polynomials $N_\Delta^{\partial,\text{trop}}$, mimicking the Caporaso-Harris formula. This formula allows one to compute them.\\

The paper is organized as follows. In the second section, we recall the standard definitions related to tropical curves. In the third section, we describe the tropical enumerative problem that defines $N_\Delta^{\partial,\text{trop}}$, we prove its invariance and state the recursive formula. The fourth section is devoted to the proof of the formula. In the last section, we provide some examples of computations using the formula.\\

\textit{Acknowledgements} The author is grateful to Ilia Itenberg for numerous discussions leading to the writing of this paper, and to Maxence Blomme for helping to implement the algorithm.\\

\thanks{
The author is partially supported by the ANR
grant ANR-18-CE40-0009 ENUMGEOM}

\section{Tropical Curves}

	\subsection{Abstract and parametrized tropical curves}

Let $\overline{\Gamma}$ be a finite connected graph without bivalent vertices. Let $\overline{\Gamma}_\infty^0$ be the set of $1$-valent vertices of $\overline{\Gamma}$. If $m$ denotes the cardinal of $\overline{\Gamma}_\infty^0$, its elements are labeled with integers from $[\![1;m]\!]$. Let $\Gamma=\overline{\Gamma}\backslash\overline{\Gamma}^0_\infty$. We denote by $\Gamma^0$ the set of vertices of $\Gamma$, and by $\Gamma^1$ the set of edges of $\Gamma$. The non-compact edges resulting from the eviction of $1$-valent vertices are called \textit{unbounded ends}. The set of unbounded ends is denoted by $\Gamma^1_\infty$. Let $l:\Gamma^1\backslash\Gamma^1_\infty\rightarrow\RR_{\geqslant 0}$ be a function, called length function. It endows $\Gamma$ with the structure of a metric graph in the following way : the bounded edges $E$ are isometric to $[0;l(E)]$, and the unbounded ends have infinite length and are isometric to $[0;+\infty[$.

\begin{defi}
Such a metric graph $\Gamma$ is called an \textit{abstract tropical curve}.
\end{defi}

An isomorphism between two abstract tropical curves $\Gamma$ and $\Gamma'$ is an isometry $\Gamma\rightarrow\Gamma'$. In particular an automorphism does not necessarily preserve the labeling. We now define parametrized tropical curves taking values in $N_\RR=N\otimes\RR$.

\begin{defi}
A \textit{parametrized tropical curve} in $N_\RR\simeq\RR^2$ is a pair $(\Gamma,h)$, where $\Gamma$ is an abstract tropical curve and $h:\Gamma\rightarrow\RR^2$ is a map satisfying the following requirements:
\begin{itemize}
\item For every edge $E\in\Gamma^1$, the map $h|_E$ is affine. If we choose an orientation of $E$, the value of the differential of $h$ taken at any interior point of $E$, evaluated on a tangent vector of unit length, is called the slope of $h$ alongside $E$. This slope must lie in $N$.
\item We have the so-called \textit{balancing condition} : at each vertex $V\in\Gamma^0$, if $E$ is an edge containing $V$, let $u_E$ be the slope of $h$ along $E$ when $E$ is oriented outside $V$; then
$$\sum_{E:\partial E\ni V} u_E=0\in N.$$
\end{itemize}
\end{defi}

Two parametrized curves $h:\Gamma\rightarrow N_\RR$ and $h':\Gamma'\rightarrow N_\RR$ are called isomorphic if there exists an isomorphism of abstract tropical curves $\varphi:\Gamma\rightarrow\Gamma'$ such that $h=h'\circ\varphi$.

\begin{rem}
We could assume that $M=N=\ZZ^2$, but the distinction is now useful since the lattice $M$ is a set of functions on the space $N_\RR$ where the tropical curves live, while $N$ is the space of the slopes of the edges of a tropical curve. Moreover, notice that we deal with tropical curves in the affine space $N_\RR$, identified with its tangent space at $0$.
\end{rem}

If $e\in\Gamma^1_\infty$ is an unbounded end of $\Gamma$, let $n_e\in N$ be the slope of $h$ alongside $e$, oriented out of its unique adjacent vertex, \textit{i.e.} toward infinity. The multiset
$$\Delta=\{n_e\in N|e\in\Gamma^1_\infty \}\subset N,$$
is called the degree of the curve. It is a multiset since an element may appear several times. Using the balancing condition, one can show that $\sum_{n_e\in\Delta}n_e=0$. We say that a parametrized curve is of degree $d$ if $\Delta=\Delta_d=\{-e_1,\dots,-e_1,,-e_2,\dots,-e_2,e_1+e_2,\dots,e_1+e_2\}$ where each vector appears $d$ times.\\

\begin{rem}
The degree $\Delta$ leads to the convex lattice polygon $P_\Delta\subset M$ as defined in the introduction : the vectors $\omega(n_e,-)$, with $\omega$ the chosen volume form on $N$, are vectors in $M$, and there is a unique convex polygon obtained by putting them on top of one another. This point of view just amounts to work with an equation of the curve rather than a parametrization. In fact, this polygon would then be the Newton polygon of a polynomial equation defining the curve. However, the knowledge of the polygon $P_\Delta$ does not allow us to recover uniquely the family $\Delta$. Nevertheless, it is the case if we assume the vectors of $\Delta$ to be primitive, \textit{i.e.} their lattice length is $1$.
\end{rem}

\begin{defi}
\begin{itemize}
\item[-] Let $\Gamma$ be an abstract tropical curve. The \textit{genus} of $\Gamma$ is its first Betti number $b_1(\Gamma)$.
\item[-] A curve is \textit{rational} if it is of genus $0$.
\item[-] A parametrized tropical curve $(\Gamma,h)$ is \textit{rational} if $\Gamma$ is rational.
\end{itemize}
\end{defi}

An abstract tropical curve is then rational if it is a tree.

	\subsection{Moment of an edge}
	
	Let $h:\Gamma\rightarrow N_\RR$ be a parametrized tropical curve. Let $e\in\Gamma^1_\infty$ be an unbounded end oriented toward infinity, directed by $n_e$. Using the identification between $N_\RR$ and its tangent space, we define the moment of $e$ as the scalar
	$$\mu_e=\omega(n_e,p)\in\RR,$$
	where $p\in e$ is any point on the edge $e$. We similarly define the moment of a bounded edge if we specify its orientation. The moment of a bounded edge is reversed when its orientation is reversed.\\
	
	Intuitively, the moment of an unbounded end is just a way of measuring its position alongside a transversal axis. Therefore, fixing the moment of an unbounded end amounts to impose on the curve to pass through some point at infinity, or equivalently one of its unbounded ends to be contained in a fixed line with the same slope. In a way, this allows us to do toric geometry in a compactification of $\RR^2$ but staying in $\RR^2$. It provides a coordinate on the components of the toric boundary without even having to introduce the concept of toric boundary in the tropical world. Following this observation, the moment has also a definition in complex toric geometry, where it corresponds to the coordinate of the intersection point of the curve with the toric divisor. Let

$$\begin{array}{rccl}
\varphi: & \CC P^1 & \dashrightarrow & \text{Hom}(M,\CC^*)=N\otimes\CC^* \\
 & t & \mapsto & \chi\prod_{1}^m(t-\alpha_i)^{n_i}. \\
 \end{array}$$
	 be a parametrized rational curve. Given dual basis $(e_1^*,e_2^*)$ of $M$, and $(e_1,e_2)$ of $N$, the parametrized curve given in coordinates is as follows. Let $n_i=a_ie_1+b_ie_2$, $a=\chi(e_1^*)$ and $b=\chi(e_2^*)$, then
	 $$\varphi(t)=\left( a\prod_1^m(t-\alpha_i)^{a_i},b\prod_1^m(t-\alpha_i)^{b_i}\right)\in (\CC^*)^2.$$
	 The vectors $n_i$ of $\Delta$ define a fan $\Sigma_\Delta\subset N_\RR$  to which is associated a toric surface $\CC\Delta$. The toric divisors $D_j$ are in bijection with the rays of the fan, which are directed by the vectors $n_i$. Moreover, the map $\varphi$ extends to the points $\alpha_i$ by sending $\alpha_i$ to a point on the toric divisor $D$ corresponding to the ray directed by $n_i$. A coordinate on $D$ is a primitive monomial $\chi^m\in M$ in the lattice of characters such that $\langle m,n_i\rangle=0$. This latter condition ensures that the monomial $\chi^m$ extends on the divisor $D$. If $n_i$ is primitive, $\iota_{n_i}\omega=\omega(n_i,-)\in M$ is such a monomial, and then the complex moment is the evaluation of the monomial at the corresponding point on the divisor:
	$$\mu_i=\big(\varphi^*\chi^{\iota_{n_i}\omega}\big)(\alpha_i).$$
	
	The Weil reciprocity law gives us the following relation between the moments :
$$\prod_{i=1}^m \mu_i=(-1)^m.$$	
	We could also prove the relation using Vi\`ete's formula. In the tropical world we have an analog called the tropical Menelaus theorem, which gives a relation between the moments of the unbounded ends of a parametrized tropical curve.
	
	\begin{prop}[Tropical Menelaus Theorem]\cite{mikhalkin2017quantum} For a parametrized tropical curve $\Gamma$ of degree $\Delta$, we have
	$$\sum_{n_e\in\Delta} \mu_e =0.$$
	\end{prop}
	
	In the tropical case (resp. in the complex case), a configuration of $m$ points on the toric divisors is said to satisfy the \textit{Menelaus condition} $\sum\mu_e=0$ (resp. $\prod\mu_e=(-1)^m$.).
	
	\subsection{Moduli space of tropical curves and refined multiplicity of a simple tropical curve}
	
	Let $(\Gamma,h)$ be a parametrized tropical curve such that $\Gamma$ is trivalent, and has no \textit{flat vertex}. A flat vertex is a vertex whose outgoing edges have their slope contained in a common line. If the vertex is trivalent, it just means that for any two outgoing edges of respective slopes $u,v$ among the three, we have $\omega(u,v)=0$. In particular, when the curve is trivalent, no edge can have a zero slope since it would imply that its extremities are flat vertices. Such a curve is called a \textit{simple nodal rational tropical curve}.\\
	
	The construction and vocabulary that follow are not strictly necessary for the understanding of the enumerative problem but are useful for the proof of the recursive formula.
	
	\begin{defi}
	The \textit{combinatorial type} of a tropical curve is the homeomorphism type of its underlying labeled graph $\Gamma$, \textit{i.e.} the labeled graph $\Gamma$ without the metric.
	\end{defi}
	
	To give a graph a tropical structure, one just needs to specify the lengths of the bounded edges. If the curve is trivalent and has $m$ unbounded ends, there are $m-3$ bounded edges, otherwise the number of bounded edges is $m-3-\text{ov}(\Gamma)$, where $\text{ov}(\Gamma)$ is the overvalence of the graph. The overvalence is given by $\sum_V \text{val}(v)-3$. Therefore, the set of curves having the same combinatorial type is homeomorphic to $\RR_{\geqslant 0}^{m-3-\text{ov}(\Gamma)}$, and the coordinates are the lengths of the bounded edges. If $\Gamma$ is an abstract tropical curve, we denote by $\text{Comb}(\Gamma)$ the set of curves having the same combinatorial type as $\Gamma$.\\
	
	For a given $\text{Comb}(\Gamma)$, the boundary of $\RR_{\geqslant 0}^{m-3-\text{ov}(\Gamma)}$ corresponds to curves for which the length of an edge is zero, and therefore corresponds to a graph having a different combinatorial type. This graph is obtained by deleting the edge with zero length and merging its extremities. We can thus glue together all the cones of the finitely many combinatorial types and obtain the \textit{moduli space $\mathcal{M}_{0,m}$ of rational tropical curves with $m$ marked points}. It is a simplicial fan of pure dimension $m-3$, and the top-dimensional cones correspond to trivalent curves. The combinatorial types of codimension $1$ are called \textit{walls}.\\
	
	Given a tropical curve $\Gamma$, if we specify the slope of every unbounded end, and the position of a vertex, we can define uniquely a parametrized tropical curve $h:\Gamma\rightarrow N_\RR$. Therefore, if $\Delta\subset N$ denotes the set of slopes of the unbounded ends, the \textit{moduli space $\mathcal{M}_0(\Delta,N_\RR)$ of parametrized rational tropical curves of degree $\Delta$} is isomorphic to $\mathcal{M}_{0,m}\times N_\RR$ as a fan, where the $N_\RR$ factor corresponds to the position of the finite vertex adjacent to the first unbounded end.\\
	
	On this moduli space, we have a well-defined evaluation map that associates to each parametrized curve the family of moments of its unbounded ends :
	$$\begin{array}{crcl}
	\text{ev} : & \mathcal{M}_0(\Delta,N_\RR) & \longrightarrow & \RR^{m-1} \\
	 & (\Gamma,h) & \longmapsto & \mu=(\mu_i)_{2\leqslant i\leqslant m}
	\end{array}.$$
	By the tropical Menelaus theorem, the moment $\mu_1$ is equal to the opposite of the sum of the other moments, hence we do not take it into account in the map. Notice that the evaluation map is linear on every cone of $\mathcal{M}_0(\Delta,N_\RR)$. Moreover, both spaces have the same dimension $m-1$. Thus, if $\Gamma$ is a trivalent curve, the restriction of $\text{ev}$ on $\text{Comb}(\Gamma)\times N_\RR$ has a determinant well-defined up to sign when $\RR^{m-1}$ and $\text{Comb}(\Gamma)\subset\mathcal{M}_{0,m}$ are both endowed with their canonical basis, and $N_\RR$ is endowed with a basis of $N$. The absolute value $m_\Gamma^\CC$ is called the \textit{complex multiplicity} of the curve, well-known to factor into the following product over the vertices of $\Gamma$:
	$$m_\Gamma^\CC=\prod_V m_V^\CC,$$
	where $m_V^\CC=|\omega(u,v)|$ if $u$ and $v$ are the slopes of two outgoing edges of $V$. The balancing condition ensures that $m_V^\CC$ does not depend on the chosen edges. This multiplicity is the one that appears in the correspondence theorem of Mikhalkin \cite{mikhalkin2005enumerative}. Notice that the simple parametrized tropical curves are precisely the point of the cones with trivalent graph and non-zero multiplicity. We finally recall the definition of the refined Block-G\"ottsche multiplicity.
	
	\begin{defi}
	The refined multiplicity of a simple nodal tropical curve is
	$$m^q_\Gamma=\prod_V [m_V^\CC]_q,$$
	where $[a]_q=\frac{q^{a/2}-q^{-a/2}}{q^{1/2}-q^{-1/2}}$ is the $q$-analog of $a$.
	\end{defi}
	
	This refined multiplicity is sometimes called the Block-G\"ottsche multiplicity and intervenes in the definition of the invariant $N_\Delta^{\partial,\text{trop}}$. Notice that the multiplicity is the same for every curve inside a given combinatorial type.

\section{Definition of the invariants and recursive formula}

	\subsection{Tropical enumerative problem}

We now turn our focus into the tropical enumerative problem that provides the refined tropical invariants $N_\Delta^{\partial,\text{trop}}$ used in \cite{mikhalkin2017quantum}. This family of enumerative problems depends on the choice of the degree $\Delta$, and the recursive formula proven in the present paper gives a relation between all these invariants.\\
	
Let $\Delta=\{v_1,\dots,v_{m}\}\subset N$ be a tropical degree. We do not assume the vectors $v_i$ to be primitive since the recursive formula almost always makes appear degrees with non-primitive vectors. For each $v_i\in\Delta$ we choose some scalar $\mu_i\in\RR$. Because of the tropical Menelaus theorem, the sum of the moments of all the unbounded ends of a tropical curve is zero, so a necessary condition for the scalars $\mu_i$ to be the moments of some tropical curve is $\sum_1^{m}\mu_i=0$. We now look at the set $\mathcal{S}(\mu)$ of parametrized rational tropical curves of degree $\Delta$ that have $\mu=(\mu_i)$ as family of moments, \textit{i.e.} $\text{ev}(\Gamma)=\mu$. The invariance amounts to prove that the count of parametrized rational curves with appropriate multiplicity does not depend on $\mu$. This tropical enumerative problem is called the $\Delta$-problem.\\

Notice that due to the linear character of the evaluation map restricted to any combinatorial type, each type contributes at most one solution unless the map is non-injective. Let $\text{Comb}(\Gamma)$ be a top-dimensional combinatorial type, for which the evaluation map is not injective. Since $\text{Comb}(\Gamma)$ and $\RR^{m-1}$ have the same dimension, the evaluation map is not surjective either. Similarly, the restriction of the evaluation map to non top-dimensional combinatorial types fails to be surjective for dimensional reasons. A family of moments $\mu$ is said to be generic if it is chosen outside the image of the evaluation map restricted to non top-dimensional combinatorial types, and top-dimensional types with non-injective evaluation map. Thus, if the configuration of moments $\mu$ is chosen generically, the set of solutions $\mathcal{S}(\mu)$ is finite, and the rational curves of $\mathcal{S}(\mu)$ solution to the problem are simple, and they have a well-defined refined tropical multiplicity $m_\Gamma^q$. Then we set
$$N_{\Delta}^{\partial,\text{trop}}(q,\mu)=\sum_{\Gamma\in\mathcal{S}(\mu)}m_\Gamma^q.$$

\begin{rem}
Let $\text{Comb}(\Gamma)$ be a combinatorial type, and $A$ the linear map associated to the restriction $\text{ev}_{\text{Comb}(\Gamma)}$ of $\text{ev}$ to $\text{Comb}(\Gamma)$. Let $(l,V)$ denotes the coordinates on $\text{Comb}(\Gamma)$ given by the length of the edges and the position of a specified vertex. To see if $\text{Comb}(\Gamma)$ contributes a solution, one just needs to solve the system $A(l,V)=\mu$, leading to a formal solution $(l,V)=A^{-1}(\mu)\in\RR^{m-3}\times N_\RR$, and check that all its first $m-3$ coordinates are non-negative.
\end{rem}

\begin{theo}
The value of $N_{\Delta}^{\partial,\text{trop}}(q,\mu)$ does not depend on $\mu$ as long as $\mu$ is generic.
\end{theo}

A proof can be found in the next subsection.

	\subsection{Proof of tropical invariance}

	The proof of invariance goes in the same way as many tropical proofs of invariance by showing that we have a local invariance of the count around the walls of the tropical moduli space. 
	
	\begin{proof}
	We choose two generic configurations $\mu(0)$ and $\mu(1)$, and choose a generic path $\mu(t)$ between them. Due to the genericity, we know that the set $F$ of values of $t$ where $\mu(t)$ meets the non-generic configurations is finite, and $N_{\Delta}^{\partial,\text{trop}}(q,\mu(t))$ is constant on the connected components of the complement of this exceptional set $F$. We now need to check that the value is constant around these special values.\\
	
	Let $t^*$ be such a special value. Thanks to the genericity of the path, it means that at least one of the curves of $\mathcal{S}(\mu(t^*))$ has a unique four-valent vertex $V$. There are three ways to deform this curve into a trivalent one by choosing a splitting of the quadrivalent vertex, meaning there are three maximal cones adjacent to the wall. In some cases, one of the deformation leads to a flat vertex, \textit{i.e.} a non-injective combinatorial type. Let $E_1,E_2,E_3,E_4$ be the adjacent edges directed by $a_1,a_2,a_3,a_4$, with ingoing orientations. Their index $i$ is taken in $\ZZ/4\ZZ$. The splittings are denoted by $12//34$, $13//24$ and $14//23$ according to the pairing of vertices. Around a wall, one curve solution may divide in two solutions, or the other way around two solutions may merge into one solution.\\
	
	We first assume that there are no parallel edges among the edges $E_i$. Let us prove that up to a relabeling we can assume :
	\begin{itemize}
	\item[-] for each $i$ we have $\omega(a_i,a_{i+1})>0$,
	\item[-] we have $\omega(a_2,a_3)>\omega(a_1,a_2)$.
\end{itemize}
The first point essentially consists in finding some cyclic counterclockwise order on the outgoing vectors $a_i$. Let us take such a cyclic order and prove that it satisfies these conditions : if we had $\omega(a_4,a_1)<0$ (same for $\omega(a_1,a_2)<0$ and other values), then because of the counterclockwise cyclic order all the vectors $a_i$ are in a half-plane and their sum would not be zero, which is absurd. So we have $\omega(a_i,a_{i+1})>0$. For the second point, the assumption that there are no parallel vectors ensures that $\omega(a_i,a_{i+1}+a_{i-1})\neq 0$ for any $i$, thus there are no consecutive equal values. Hence we can assume that $\omega(a_2,a_3)>\omega(a_1,a_2)$ up to a cyclic shift of the indices.\\

Let us notice that
\begin{align*}
\omega(a_4,a_1) &=\omega(-a_1-a_2-a_3,a_1) \\
	&=\omega(a_1,a_2+a_3)>0,\\
\text{and }	\omega(a_3,a_4) &=\omega(a_3,-a_1-a_2-a_3) \\
	&=\omega(a_1+a_2,a_3)>0.\\
\end{align*}
	
	To prove the local invariance, we need to know the repartition of the combinatorial types around the wall, that is, the adjacent combinatorial types providing a solution when $\mu(t)$ moves slightly. Using the correspondence theorem of Mikhalkin \cite{mikhalkin2005enumerative} or the tropical proof of invariance of the count with complex multiplicities given by A. Gathmann-H. Markwig in \cite{gathmann2007numbers}, this repartition is known to match the equality given between complex multiplicities $m_\Gamma=\prod_Vm_V^\CC>0$. All the vertices in the respective products for the three adjacent combinatorial types are the same, except the two vertices resulting from the splitting of the quadri-valent vertex. The desired relation is then
	$$\begin{array}{ccccccc}
	\omega(a_1,a_2)\omega(a_1+a_2,a_3) & + & \omega(a_1,a_3)\omega(a_2,a_1+a_3) & + & \omega(a_2,a_3)\omega(a_2+a_3,a_1) & = &0, \\
\text{ for }12//34 & & \text{ for }13//24 & & \text{ for }14//23 & & \\
	\end{array}$$
	and the repartition of combinatorial types around the wall is given by the sign of each term. It means up to sign that one is positive and is on one side of the wall, and the two other ones are negative, on the other side of the wall. Hence, we just need to study the signs of each term to know which curve is on which side. We know that $\omega(a_1,a_2)$ and $\omega(a_1+a_2,a_3)$ are positive, therefore their product, which is the term of $12//34$, is also positive. We know that $\omega(a_2,a_3)$ is positive, but $\omega(a_2+a_3,a_1)$ is negative, therefore their product is negative and $14//23$ is on the other side of the wall. It means that the combinatorial types $12//34$ and $14//23$ are on opposite sides of the wall. We need to determine on which side the type $13//24$ is, and that is given by the sign of the middle term. As by assumption $\omega(a_2,a_1+a_3)=\omega(a_2,a_3)-\omega(a_1,a_2)>0$, it is determined by the sign of $\omega(a_1,a_3)$.\\
	\begin{itemize}
	\item If $\omega(a_1,a_3)>0$, then $12//34$ and $13//24$ are on the same side, and the invariance for refined multiplicities is dealt with the identity
	\begin{align*}
 & (q^{\omega(a_2,a_3)}-q^{-\omega(a_2,a_3)})(q^{\omega(a_1,a_2+a_3)}-q^{-\omega(a_1,a_2+a_3)}) \\
 = & (q^{\omega(a_1,a_2)}-q^{-\omega(a_1,a_2)})(q^{\omega(a_1+a_2,a_3)}-q^{-\omega(a_1+a_2,a_3)}) \\
+ & (q^{\omega(a_1,a_3)}-q^{-\omega(a_1,a_3)})(q^{\omega(a_2,a_1+a_3)}-q^{-\omega(a_2,a_1+a_3)}),\\
\end{align*}
	\item and if $\omega(a_1,a_3)<0$, then $14//23$ and $13//24$ are on the same side and then the invariance for refined multiplicities is true since

\begin{align*}
 & (q^{\omega(a_2,a_3)}-q^{-\omega(a_2,a_3)})(q^{\omega(a_1,a_2+a_3)}-q^{-\omega(a_1,a_2+a_3)}) \\
+ & (q^{\omega(a_3,a_1)}-q^{-\omega(a_3,a_1)})(q^{\omega(a_2,a_1+a_3)}-q^{-\omega(a_2,a_1+a_3)})\\
 = & (q^{\omega(a_1,a_2)}-q^{-\omega(a_1,a_2)})(q^{\omega(a_1+a_2,a_3)}-q^{-\omega(a_1+a_2,a_3)}). \\
\end{align*}

	\end{itemize}
	
If we have some edges parallel among the vectors $a_i$, either two consecutives vectors are parallel, and then the invariance is straightforward, since there are only two adjacent combinatorial types with equal non-zero multiplicity, or we can choose a cyclic labeling such that $a_1$ and $a_3$ are parallel. We then have $\omega(a_1,a_3)=0$. It means that one of the determinant multiplicities is zero, which is normal since the associated combinatorial type would have a flat vertex. Thus,
$$\begin{array}{ccccc}
	\omega(a_1,a_2)\omega(a_1+a_2,a_3)  & + & \omega(a_2,a_3)\omega(a_2+a_3,a_1) & = &0. \\
\text{ for }12//34  & & \text{ for }14//23 & & \\
	\end{array}$$
	It means that the two terms are of opposite sign. Assume the first one is positive, and thus $\omega(a_1,a_2)$ and $\omega(a_1+a_2,a_3)=\omega(a_2,a_3)$ have the same sign. The refined multiplicity is then
	$$(q^{\omega(a_1,a_2)}-q^{-\omega(a_1,a_2)})(q^{\omega(a_2,a_3)}-q^{-\omega(a_2,a_3)}).$$
	The second term being negative, it means that $\omega(a_2,a_3)$ and $\omega(a_1,a_2+a_3)=\omega(a_1,a_2)$ have the same sign. The refined multiplicity is the the same and we have the desired local invariance. It is the same if the first term is negative.
\end{proof}

\begin{rem}
The technicalities in the proof are just needed to find the repartition of the combinatorial types around the wall. This repartition could also be found by looking at the subdivisions of the Newton polygon, which are dual to the tropical curves. The quadrilateral dual to the quadrivalent vertex has three subdivisions (resp. two in the case of a flat vertex) matching the three splittings of the vertex: one using the big diagonal, one using the small diagonal, and one by using a parallelogram. Then, we can show that the repartition is given by putting the subdivision using the big diagonal alone on one side of the wall. See \cite{itenberg2013block}.
\end{rem}

	\subsection{Recursive formula}
	
	Before stating the formula we need to introduce some notations. Let $\Delta=\{v_1,v_2,\dots,v_m\}\subset N\simeq\ZZ^2$ be a degree for plane curves, \textit{i.e.} a family of vectors that whose sum is zero. Let $\mu\in\RR^m$ be the family of moments satisfying the tropical Menelaus condition, and let $\Gamma$ be a parametrized tropical curve in $\mathcal{S}(\mu)$. As $\Gamma$ is a tree, there is a unique shortest path between the edges directed by $v_1$ and $v_m$, which we call a \textit{chord}. The chord has a natural orientation from the end $v_1$ to the end $v_m$. Once we remove the chord, the curve $\Gamma$ disconnects into several components $\Gamma_1,\dots,\Gamma_p$, indexed in the order in which they meet the chord. Let $E_i$ be the edge in $\Gamma_i$ adjacent to the chord, and $V_i$ the vertex in which both meet. All these notations along with the ones which are about to follow are depicted on Figture \ref{skeleton}.\\

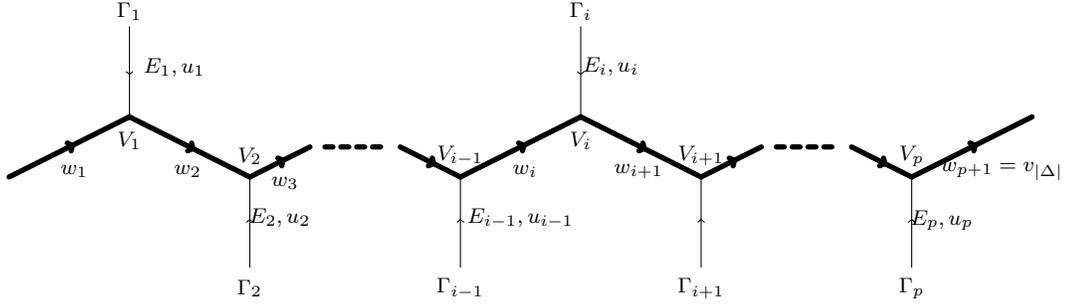
\begin{figure}
\begin{center}
\begin{tikzpicture}[line cap=round,line join=round,>=triangle 45,x=0.4cm,y=0.4cm]

\draw [line width=2.pt] (-1.,2.)-- (3.,4.);
\draw [line width=2.pt] (1.1072776923231789,3.053638846161589) -- (1.064366615393908,2.871266769212186);
\draw [line width=2.pt] (1.1072776923231789,3.053638846161589) -- (0.935633384606092,3.1287332307878137);
\draw [line width=2.pt] (3.,4.)-- (7.,2.);
\draw [line width=2.pt] (5.107277692323179,2.9463611538384105) -- (4.935633384606092,2.871266769212186);
\draw [line width=2.pt] (5.107277692323179,2.9463611538384105) -- (5.064366615393908,3.1287332307878137);
\draw [line width=2.pt] (7.,2.)-- (9.,3.);
\draw [line width=2.pt] (8.10727769232318,2.55363884616159) -- (8.064366615393908,2.3712667692121867);
\draw [line width=2.pt] (8.10727769232318,2.55363884616159) -- (7.935633384606093,2.6287332307878146);
\draw [line width=2.pt] (12.,3.)-- (14.,2.);
\draw [line width=2.pt] (13.10727769232318,2.4463611538384114) -- (12.935633384606092,2.3712667692121867);
\draw [line width=2.pt] (13.10727769232318,2.4463611538384114) -- (13.064366615393908,2.6287332307878146);
\draw [line width=2.pt] (14.,2.)-- (18.,4.);
\draw [line width=2.pt] (16.10727769232318,3.053638846161589) -- (16.064366615393908,2.871266769212186);
\draw [line width=2.pt] (16.10727769232318,3.053638846161589) -- (15.935633384606092,3.1287332307878137);
\draw [line width=2.pt] (18.,4.)-- (22.,2.);
\draw [line width=2.pt] (20.10727769232318,2.9463611538384105) -- (19.935633384606092,2.871266769212186);
\draw [line width=2.pt] (20.10727769232318,2.9463611538384105) -- (20.064366615393908,3.1287332307878137);
\draw [line width=2.pt] (22.,2.)-- (24.,3.);
\draw [line width=2.pt] (23.10727769232318,2.55363884616159) -- (23.064366615393908,2.3712667692121867);
\draw [line width=2.pt] (23.10727769232318,2.55363884616159) -- (22.935633384606092,2.6287332307878146);
\draw [line width=2.pt] (27.,3.)-- (29.,2.);
\draw [line width=2.pt] (28.10727769232318,2.4463611538384114) -- (27.935633384606092,2.3712667692121867);
\draw [line width=2.pt] (28.10727769232318,2.4463611538384114) -- (28.064366615393908,2.6287332307878146);
\draw [line width=2.pt] (29.,2.)-- (33.,4.);
\draw [line width=2.pt] (31.10727769232318,3.053638846161589) -- (31.064366615393908,2.871266769212186);
\draw [line width=2.pt] (31.10727769232318,3.053638846161589) -- (30.935633384606092,3.1287332307878137);
\draw [line width=2.pt,dash pattern=on 3pt off 3pt] (9.5,3.)-- (11.5,3.);
\draw [line width=2.pt,dash pattern=on 3pt off 3pt] (24.5,3.)-- (26.5,3.);
\draw (3.,7.)-- (3.,4.);
\draw (3.,5.416041925623623) -- (2.8920539043732294,5.5);
\draw (3.,5.416041925623623) -- (3.107946095626771,5.5);
\draw (7.,-1.)-- (7.,2.);
\draw (7.,0.5839580743763765) -- (7.107946095626771,0.5);
\draw (7.,0.5839580743763765) -- (6.892053904373229,0.5);
\draw (14.,-1.)-- (14.,2.);
\draw (14.,0.5839580743763765) -- (14.10794609562677,0.5);
\draw (14.,0.5839580743763765) -- (13.89205390437323,0.5);
\draw (18.,7.)-- (18.,4.);
\draw (18.,5.416041925623623) -- (17.89205390437323,5.5);
\draw (18.,5.416041925623623) -- (18.107946095626772,5.5);
\draw (22.,-1.)-- (22.,2.);
\draw (22.,0.5839580743763765) -- (22.107946095626772,0.5);
\draw (22.,0.5839580743763765) -- (21.89205390437323,0.5);
\draw (29.,-1.)-- (29.,2.);
\draw (29.,0.5839580743763765) -- (29.107946095626772,0.5);
\draw (29.,0.5839580743763765) -- (28.89205390437323,0.5);
\begin{scriptsize}
\draw[color=black] (3.,3.2) node {$V_1$};
\draw[color=black] (7.,2.7) node {$V_2$};
\draw[color=black] (18.,3.2) node {$V_i$};
\draw[color=black] (22.,2.7) node {$V_{i+1}$};
\draw[color=black] (14.,2.7) node {$V_{i-1}$};
\draw[color=black] (29,2.7) node {$V_p$};
\draw[color=black] (1.1642969037197943,2.2) node {$w_1$};
\draw[color=black] (4.938412247115037,2.2) node {$w_2$};
\draw[color=black] (8.168799108834694,1.8) node {$w_3$};
\draw[color=black] (16.14881417813226,2.2) node {$w_i$};
\draw[color=black] (19.922929521527504,2.2) node {$w_{i+1}$};

\draw[color=black] (3.,7.5) node {$\Gamma_1$};
\draw[color=black] (7.,-1.7) node {$\Gamma_2$};
\draw[color=black] (14.,-1.7) node {$\Gamma_{i-1}$};
\draw[color=black] (18.,7.5) node {$\Gamma_i$};
\draw[color=black] (22.,-1.7) node {$\Gamma_{i+1}$};
\draw[color=black] (29.,-1.7) node {$\Gamma_p$};
\draw[color=black] (32,2.2) node {$w_{p+1}=v_{|\Delta|}$};
\draw[color=black] (4.5,5.630077747790147) node {$E_1,u_1$};
\draw[color=black] (8,0.6405693277083074) node {$E_2,u_2$};
\draw[color=black] (16,0.6405693277083074) node {$E_{i-1},u_{i-1}$};
\draw[color=black] (19,5.678053790290934) node {$E_i,u_i$};
\draw[color=black] (30,0.5606092568736626) node {$E_p,u_p$};
\end{scriptsize}
\end{tikzpicture}

\caption{skeleton of the curve $\Gamma$ with every notation. The chord is in fat.\label{skeleton}}

\end{center}
\end{figure}

There are two possibilities for $\Gamma_i$ : 
\begin{itemize}
\item[-] either $E_i$ is an unbounded edge, directed by some $v_i\in\Delta$, and we set $u_i=-v_i$,
\item[-] or $E_i$ is bounded and $\Gamma_i$ contains more than one unbounded edge of $\Gamma$. We then denote by $\widetilde{\Delta_i}\subset\Delta$ the set of directing vectors of unbounded ends of $\Gamma$ that belong to $\Gamma_i$. Let $u_i=-\sum_{v\in\widetilde{\Delta_i}}v$ be the directing vector of $E_i$ going toward the chord, and $\Delta_i=\widetilde{\Delta_i}\sqcup\{u_i\}$ be the degree of the curve $\Gamma_i'$ obtained by letting $E_i$ going to infinity instead of stopping when meeting the chord in $V_i$.
\end{itemize}
Finally, let $w_i$ be the vector directing the edge of the chord between $V_{i-1}$ and $V_i$. This means that $w_1=-v_1$ and $w_{i+1}=w_i+u_i$. Let also $\sigma_i=\omega(w_i,w_{i+1})$ be the signed multiplicity of the vertex $V_i$. We now can derive a recursive formula from this description.

	\begin{theo}
	With the above notation, we have
	$$N_\Delta^{\partial,\text{trop}}(q)=\sum_\ast \prod_{i=1}^p [\sigma_i]_q N_{\Delta_i}^{\partial,\text{trop}}(q),$$
	where the sum $\ast$ is over the ordered partitions of $\Delta-\{v_1,v_m\}$ into
	$$\Delta-\{v_1,v_m\}=\bigsqcup_{i=1}^p \widetilde{\Delta_i} , \ 
	u_{i}=-\sum_{v\in\widetilde{\Delta_i}}v, $$
	 such that
$$\left\{ \begin{array}{l}
\sigma_i>0\Rightarrow |\widetilde{\Delta_i}|=1 \\
\omega(\sigma_i u_i,\sigma_{i+1} u_{i+1})\geqslant 0, \\
\end{array}\right.$$
and up to a reordering of consecutive indices $i$ having respective colinear vectors $u_i$.
	\end{theo}
	
\begin{rem}
The term \textit{"ordered partition"} means that the set is subdivided into several subsets, but we keep track of their order by labeling them. Ordered partitions in $p$ subsets are thus in bijection with the surjections to $[\![1;p]\!]$. The reordering means that orders that differ by a sequence of permutations of consecutive indices $i$ and $i+1$ such that $u_i$ and $u_{i+1}$ are colinear, are counted only once.
\end{rem}


\section{Proof of the recursive formula}

To prove the recursive formula, we find a way to describe the curves solution to the problem for a specific value of $\mu$. The idea is to choose an idealistic configuration of constraints and then a $1$-parameter family $\mu(t)$ of moments getting closer and closer to this idealistic but unreachable configuration. Such a $1$-parameter family $\mu(t)$ is called a deformation. Then, we describe the specific combinatorial types that continue to provide a solution through the deformation process toward this ideal configuration.\\

\begin{rem}
The same idea drives the tropical proof of the Caporaso-Harris formula in \cite{gathmann2007caporaso} : one deforms the constraints by making one of the marked points going to infinity on the left. The only combinatorial types that "survive" the deformation (see definition below) are those that either have the corresponding marked point on a horizontal edge of the curve, or that split into a \textit{floor} containing the marked point and a curve of lower degree, joined to the floor by horizontal edges. We implement this ideology in our setting to provide a way of computing the invariants $N_\Delta^{\partial,\text{trop}}$.
\end{rem}

First, recall the evaluation map $\text{ev}:\mathcal{M}_0(\Delta,N_\RR)\rightarrow\RR^{m-1}.$ If $\mu\in\RR^{m-1}$ is a family of moments for the last $m-1$ ends of the curves, we say that a parametrized tropical curve is solution to the $\Delta$-problem with value $\mu$ if $\text{ev}(\Gamma)=\mu$. In some cases we see $\mu$ as a function $\Delta-\{v_1\}\rightarrow\RR$ that assigns to any unbounded end its moment. If $\widetilde{\Xi}\subset\Delta-\{v_1\}$ is a subset, then
$$\Xi=\{-\sum_{v\in\widetilde{\Xi}}v\}\sqcup\widetilde{\Xi}$$
is still a tropical degree and this notation allows us to consider the $\Xi$-problem with value $\mu|_{\widetilde{\Xi}}$.

\begin{defi}
We call a \textit{deformation vector} a lattice vector $\delta\in\ZZ^{m-1}\subset\RR^{m-1}$. The associated \textit{deformation} of an element $\mu\in\RR^{m-1}$ is the half-line $\mu+\RR_{\geqslant 0}\delta$, parametrized by $t\mapsto\mu+t\delta$.
\end{defi}

\begin{defi}
Let $\Gamma$ be a tropical curve with non-zero multiplicity. On the orthant $\text{Comb}(\Gamma)\times N_\RR$ of curves having the same combinatorial type, the evaluation map is linear with matrix $A$ in the canonical basis :
$$A=\text{ev}|_{\text{Comb}(\Gamma)\times N_\RR}:\RR_{\geqslant 0}^{m-3}\times N_\RR\rightarrow\RR^{m-1}.$$
If $\delta$ is a deformation vector, we say that the combinatorial type of $\Gamma$ \textit{survives} the deformation if the first $m-3$ coordinates of $A^{-1}\delta$ are non-negative.
\end{defi}

\begin{rem}
As the multiplicity of the curve is non-zero, the matrix $A$ is invertible. Hence, any small deformation of the image can be pull back by the evaluation map to a small deformation of the curve, which means a variation of the length of the bounded edges and maybe a translation. The assumption that the first coordinates are non-negative means that going along the deformation $\delta$, the length of the edges are non-decreasing, and the half-line $\text{ev}|_{\text{Comb}(\Gamma)\times N_\RR}^{-1}(\mu)+\RR_{\geqslant 0}A^{-1}\delta$ does not meet the boundary of the orthant, place where the deformation of the curve cannot go on since an edge has zero length and a quadri-valent vertex appears.
\end{rem}

Let $\Gamma$ be a tropical curve solution to the $\Delta$-problem with value $\mu$, with non-zero-multiplicity, and $\delta$ a deformation vector. Let $f$ be some affine function defined on the orthant of the combinatorial type of $\Gamma$, with linear part $\overline{f}$. Then we write
$$\der{f}=\overline{f}(A^{-1}\delta),$$
for the variation of $f$ along the deformation. The functions of interest are the position of a vertex $V$, the length of a bounded edge $E$, and the moment of some edge $E$.\\

\begin{expl}
If $f=V:\text{Comb}(\Gamma)\times N_\RR\rightarrow N_\RR$ is the position of a vertex of $\Gamma$, then $\der{V}$ is the direction in which $V$ moves when the curve is deformed by making $\text{ev}(\Gamma)$ go in direction $\delta$.
\end{expl}
	
From now on, we consider the deformation vector $\delta=(0,\dots,0,-1)$, which means that the moment of the edge directed by $v_{m}$ goes to $-\infty$, and thus the moment of the edge directed by $v_1$ goes to $+\infty$. We look for combinatorial types that survive this deformation.

\begin{prop}
For $t$ large enough, the only combinatorial types that contribute a solution to $N_\Delta^{\partial,\text{trop}}(q,\mu+t\delta)$ are surviving combinatorial types.
\end{prop}

\begin{proof}
Each combinatorial type  $\text{Comb}(\Gamma)$ of tropical curves provides a formal solution to the problem, meaning that we can solve $\text{ev}|_{\text{Comb}(\Gamma)\times N_\RR}(l,V))=\mu$ formally and find the lengths of the edges, but some of them might be negative. The formal solution is a true solution if the length of each edge is non-negative. If $\text{Comb}(\Gamma)$ is not a surviving combinatorial type, the length of some edge strictly decreases when $t$ increases, and it becomes negative if $t$ is large enough, therefore the combinatorial type no longer provides a solution.
\end{proof}

\begin{rem}
As the length of some edges might be constant through the deformation, the survival property is not enough to ensure a combinatorial type ultimately provide a true solution. More precisely, among the combinatorial types differing from one another by a permutation of consecutive indices $i$ having colinear directing vectors $u_i$, exactly one ultimately provides a true solution. This is the place where the reordering appears.
\end{rem}

Using the balancing condition, we see that the moment $\mu_{E_i}$ of $E_i$ is constant equal to minus the sum of coordinates of $\mu|_{\widetilde{\Delta_i}}$. This means that the edge $E_i$ is contained in a fixed line. The vertices of $\Gamma_i$ are fixed because the moment of two of their incident edges are constant. Thus, if we change the moment of $v_1$ and $v_m$, the only way the edges $E_i$ can move is by varying their length while the edges in each $\Gamma_i$ different from $E_i$ are fixed. Moreover, only their extremity $V_i$ in which $\Gamma_i$ meets the chord can move, and these vertices move on the lines that respectively contain $E_i$.\\


For the combinatorial type of $\Gamma$ to survive the deformation, we need to check that neither the length of the edges of the chord nor the length of the bounded edges $E_i$ go to $0$. Recall $\sigma_i=\omega(w_i,w_{i+1})$ the signed multiplicity of the vertex $V_i$. If $\sigma_i<0$, that means that at $V_i$, the chord turns right, as we can see on Figure \ref{right vertex}, and if $\sigma_i>0$, it means that the chord turns left, as we can see on Figure \ref{left vertex}.\\


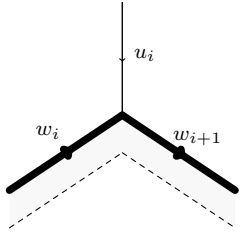
\begin{figure}
\begin{center}

\begin{tikzpicture}[line cap=round,line join=round,>=triangle 45,x=0.5cm,y=0.5cm]

\clip(0.,0.) rectangle (8.,7.);
\fill[line width=2.8pt,dash pattern=on 2pt off 2pt,color=lightgray,fill=lightgray,fill opacity=0.1] (1.,1.) -- (1.,2.) -- (4.,4.) -- (7.,2.) -- (7.,1.) -- (4.,3.) -- cycle;
\draw [dash pattern=on 2pt off 2pt] (1.,1.)-- (4.,3.);
\draw [dash pattern=on 2pt off 2pt] (4.,3.)-- (7.,1.);
\draw [line width=2.8pt] (1.,2.)-- (4.,4.);
\draw [line width=2.8pt] (2.5779778340713233,3.051985222714216) -- (2.5606494264999187,2.9090258602501224);
\draw [line width=2.8pt] (2.5779778340713233,3.051985222714216) -- (2.439350573500081,3.090974139749878);
\draw [line width=2.8pt] (4.,4.)-- (7.,2.);
\draw [line width=2.8pt] (5.577977834071324,2.9480147772857848) -- (5.439350573500081,2.9090258602501233);
\draw [line width=2.8pt] (5.577977834071324,2.9480147772857848) -- (5.560649426499919,3.090974139749878);
\draw (4.,7.)-- (4.,4.);
\draw (4.,5.445331345731765) -- (3.929711730226555,5.5);
\draw (4.,5.445331345731765) -- (4.070288269773445,5.5);
\draw (4.,7.)-- (4.,4.);
\begin{scriptsize}
\draw[color=black] (2.0268098916047643,3.539672008973914) node {$w_i$};
\draw[color=black] (6.0046053069315795,3.3938889309252875) node {$w_{i+1}$};
\draw[color=black] (4.6,5.591048178658161) node {$u_i$};
\end{scriptsize}
\end{tikzpicture}

\caption{Vertex that turns right, i.e. $\sigma_i<0$.\label{right vertex}}
\end{center}
\end{figure}

\begin{figure}
\begin{center}

\begin{tikzpicture}[line cap=round,line join=round,>=triangle 45,x=0.5cm,y=0.5cm]
\clip(0.,0.) rectangle (7.,7.);
\fill[line width=2.8pt,dash pattern=on 2pt off 2pt,color=lightgray,fill=lightgray,fill opacity=0.1] (2.,1.) -- (1.,2.) -- (4.,4.) -- (5.,7.) -- (6.,6.) -- (5.,3.) -- cycle;
\draw [dash pattern=on 2pt off 2pt] (2.,1.)-- (5.,3.);
\draw [dash pattern=on 2pt off 2pt] (5.,3.)-- (6.,6.);
\draw [line width=2.8pt] (1.,2.)-- (4.,4.);
\draw [line width=2.8pt] (2.5779778340713237,3.051985222714216) -- (2.5606494264999182,2.909025860250123);
\draw [line width=2.8pt] (2.5779778340713237,3.051985222714216) -- (2.4393505735000818,3.090974139749877);
\draw [line width=2.8pt] (4.,4.)-- (5.,7.);
\draw [line width=2.8pt] (4.529636136703527,5.588908410110582) -- (4.603726478462344,5.465424507179218);
\draw [line width=2.8pt] (4.529636136703527,5.588908410110582) -- (4.396273521537655,5.534575492820782);
\draw (7.,1.)-- (4.,4.);
\draw (5.4613434238485885,2.5386565761514115) -- (5.5497013121946726,2.5497013121946717);
\draw (5.4613434238485885,2.5386565761514115) -- (5.450298687805328,2.450298687805328);
\begin{scriptsize}
\draw[color=black] (2.026809891604764,3.5396720089739158) node {$w_i$};
\draw[color=black] (3.8,5.98674510479015) node {$w_{i+1}$};
\draw[color=black] (6.,2.706625848696049) node {$u_i$};
\end{scriptsize}
\end{tikzpicture}

\caption{Vertex that turns left, i.e. $\sigma_i>0$.\label{left vertex}}
\end{center}
\end{figure}
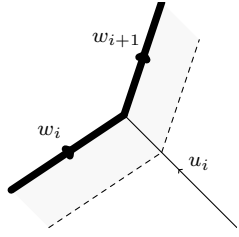

Let $\tau_i=\omega(w_i,V_i)=\omega(w_i,V_{i-1})$ be the moment of the edge directed by $w_i$. The balancing condition ensures that
$$\tau_{i+1}=\tau_i+\mu_{E_i},$$
hence all the moments $\tau_i$ only differ from one another by a constant, and thus all go to $-\infty$ through the deformation process, since $\tau_{p+1}$ is the moment of the edge directed by $v_m$ that goes to $-\infty$ with velocity $-1$ by assumption. We then have $\der{\tau_i}=-1$.\\

We now write down some equations whose derivation allows us to obtain the variations of the lengths of the bounded edges of the curve. We separate the case of edges which are adjacent to the chord from the edges which are part of the chord. We denote by a dependence on $t$ the fact that the functions (position of a vertex, moment of an edge, length of an edge) are taken on the curve of $\text{Comb}(\Gamma)\times N_\RR$ whose evaluation is $\mu+t\delta$, in case the orthant provides a true solution. As $t=0$ provides a true solution, the formal solutions are true solutions at least for small values of $t$. The formal solution os true for any $t$ if the combinatorial type is a surviving one.\\

\begin{itemize}
\item First, let $P_i$ be a fixed point on $E_i$, which is the other extremity if $E_i$ is bounded and any point otherwise. We have
$$V_i(t)=P_i+l_i(t)u_i=P_i+l_i(t)(w_{i+1}-w_i),$$
where $l_i(t)$ is the length of the edge between $V_i$ and $P_i$. Hence,
\begin{align*}
\tau_i=\omega(w_i,V_i(t)) & =\omega(w_i,P_i)+l_i(t)\omega(w_i,w_{i+1}-w_i) \\
 & =\omega(w_i,P_i)+l_i(t)\sigma_i. \\
\end{align*}
Thus, by derivating, we get $-1=\frac{\dd\tau_i}{\dd t}=\sigma_i\frac{\dd l_i}{\dd t}$. This means that :
\begin{itemize}
\item If at $V_i$ the chord turns left ($\sigma_i>0$), then $l_i$ decreases as $t$ goes to $+\infty$, the vertex $V_i$ goes up the edge $E_i$ and will meet a vertex if there is one, which is the case if and only if $E_i$ is bounded.
\item If at $V_i$ the chord turns right ($\sigma_i<0$), then $l_i$ increases as $t$ goes to $+\infty$.
\end{itemize}


\begin{figure}
\begin{center}

\begin{tikzpicture}[line cap=round,line join=round,>=triangle 45,x=1.0cm,y=1.0cm]
\clip(1.,1.) rectangle (9.,9.);
\fill[color=lightgray,fill=lightgray,fill opacity=0.1] (1.,4.) -- (1.,3.) -- (5.,3.) -- (7.,5.) -- (9.,5.) -- (9.,6.) -- (6.,6.) -- (4.,4.) -- cycle;
\draw [dash pattern=on 2pt off 2pt] (1.,3.)-- (5.,3.);
\draw [dash pattern=on 2pt off 2pt] (5.,3.)-- (7.,5.);
\draw [line width=2.8pt] (1.,4.)-- (4.,4.);
\draw [line width=2.8pt] (2.587098033697032,4.) -- (2.5,3.8983856273534627);
\draw [line width=2.8pt] (2.587098033697032,4.) -- (2.5,4.101614372646536);
\draw [line width=2.8pt] (4.,4.)-- (6.,6.);
\draw [line width=2.8pt] (5.061587610255186,5.061587610255185) -- (5.071852211964385,4.928147788035616);
\draw [line width=2.8pt] (5.061587610255186,5.061587610255185) -- (4.928147788035617,5.071852211964383);
\draw (7.,1.)-- (4.,4.);
\draw (5.464073894017808,2.5359261059821923) -- (5.5461907076913874,2.5461907076913897);
\draw (5.464073894017808,2.5359261059821923) -- (5.453809292308612,2.453809292308612);
\draw (4.,8.)-- (6.,6.);
\draw (5.035926105982192,6.964073894017807) -- (4.953809292308611,6.95380929230861);
\draw (5.035926105982192,6.964073894017807) -- (5.046190707691389,7.046190707691388);
\draw [line width=2.8pt] (6.,6.)-- (9.,6.);
\draw [line width=2.8pt] (7.587098033697032,6.) -- (7.5,5.898385627353463);
\draw [line width=2.8pt] (7.587098033697032,6.) -- (7.5,6.101614372646536);
\draw [dash pattern=on 2pt off 2pt] (7.,5.)-- (9.,5.);
\draw (6.,6.)-- (7.,5.);
\begin{scriptsize}
\draw[color=black] (3.8128350614391566,4.394116645787751) node {$V_{i-1}$};
\draw[color=black] (2.2160377769935726,4.336051289989729) node {$w_{i-1}$};
\draw[color=black] (4.6354276019111245,5.38122769435411) node {$w_i$};
\draw [fill=uuuuuu] (4.,4.) circle (1.5pt);
\draw[color=uuuuuu] (4.064451603230582,4.132822544696656) node {$I$};
\draw[color=black] (5.9515756666662725,2.787641802042499) node {$u_{i-1}$};
\draw[color=black] (6.2,6.2) node {$V_i$};
\draw[color=black] (4.3,6.919959623001671) node {$u_i$};
\draw[color=black] (7.703213899906579,6.41672653941882) node {$w_{i+1}$};
\end{scriptsize}
\end{tikzpicture}

\caption{Deformation of an edge of the chord.\label{edge chord}}
\end{center}
\end{figure}
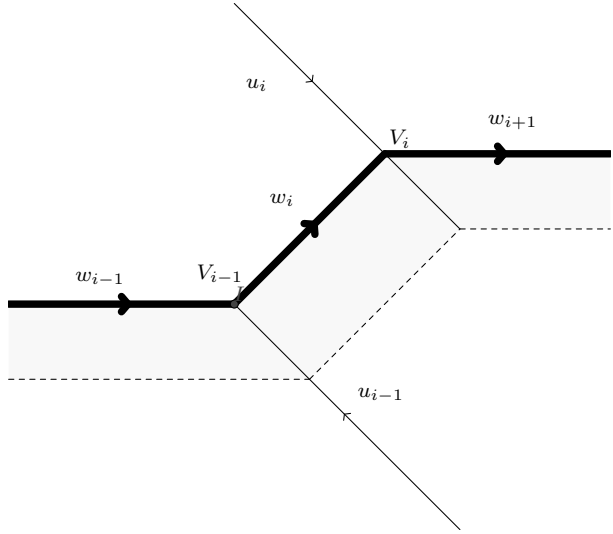

\item We consider the edge between the vertices $V_{i-1}$ and $V_i$. Let $\lambda_i$ be its length so that we have
$$V_i(t)-V_{i-1}(t)=\lambda_i(t)w_i.$$
By derivating with respect to $t$ and using previous notations, we get
$$\frac{\dd l_i}{\dd t}(w_{i+1}-w_{i})-\frac{\dd l_{i-1}}{\dd t}(w_{i}-w_{i-1}) = \frac{\dd \lambda_i}{\dd t}w_{i},$$
which is equivalent to
$$-\frac{w_{i+1}-w_{i}}{\sigma_i}+\frac{w_{i}-w_{i-1}}{\sigma_{i-1}}=\frac{\dd \lambda_i}{\dd t}w_{i}$$
since $\frac{\dd l_i}{\dd t}=-\frac{1}{\sigma_i}$. Multiplying by $\sigma_i\sigma_{i-1}$ we get
$$\sigma_{i-1}(w_{i+1}-w_{i})-\sigma_i(w_{i}-w_{i-1})=-\sigma_i\sigma_{i-1}\frac{\dd \lambda_i}{\dd t}w_{i}.$$
At this point we can check that $\sigma_{i-1}w_{i+1}+\sigma_i w_{i-1}$ is indeed colinear to $w_i$ :
\begin{align*}
\omega(w_i,\sigma_{i-1}w_{i+1}+\sigma_i w_{i-1}) & = \sigma_{i-1}\omega(w_i,w_{i+1})+\sigma_i \omega(w_i,w_{i-1}) \\
 & = \sigma_{i-1}\sigma_i-\sigma_i\sigma_{i-1}\\
 & = 0.\\
\end{align*}
In order to check the sign of $\frac{\dd \lambda_i}{\dd t}$, we can evaluate any linear form on this vector equation, for instance $\omega(w_{i-1},-)$, which gives us
$$\sigma_{i-1}^2\sigma_i\frac{\dd \lambda_i}{\dd t} = \sigma_{i-1}^2+\sigma_i\sigma_{i-1}-\sigma_{i-1}\omega(w_{i-1},w_{i+1}).$$
By noticing that
\begin{align*}
\omega(u_{i-1},u_i) & =\omega(w_i-w_{i-1},w_{i+1}-w_i)\\
&=\sigma_i+\sigma_{i-1}-\omega(w_{i-1},w_{i+1}),\\
\end{align*}
after dividing by $\sigma_{i-1}$, we are left with
$$\sigma_{i-1}\sigma_i\frac{\dd \lambda_i}{\dd t}=\omega(u_{i-1},u_i).$$
Hence, $\frac{\dd \lambda_i}{\dd t}$ is non-negative if and only if $\omega(\sigma_{i-1}u_{i-1},\sigma_i u_i)$ is non-negative.
\end{itemize}

We now can describe the conditions for a combinatorial type to survive our deformation.

\begin{prop}
Let $\Gamma$ be a parametrized tropical curve. In the above notations, the combinatorial type of $\Gamma$ survives the deformation $t\rightarrow +\infty$ if and only if
\begin{itemize}
\item[-] the edge $E_i$ is an unbounded edge whenever $\sigma_i>0$;
\item[-] for each $i$, we have $\omega(\sigma_{i-1}u_{i-1},\sigma_i u_i)\geqslant 0$.
\end{itemize}
\end{prop}

\begin{proof}
The statement follows from the previous description: a combinatorial type survives the deformation if and only if the length of the bounded edges is non-decreasing along the deformation. All the cases have previously been studied:
\begin{itemize}
\item[-] The length of the bounded edges of the chord is non-decreasing, hence for each $i$ we have $\omega(\sigma_{i-1}u_{i-1},\sigma_i u_i)\geqslant 0$.
\item[-] The bounded edges inside some $\Gamma_i$ but different from $E_i$ are constant.
\item[-] The edges $E_i$ have a non-decreasing length unless $\sigma_i>0$, and then we need the edge to be unbounded.
\end{itemize}
\end{proof}

We can now prove the recursive formula.

\begin{proof}
Let $\mu\in\RR^{m-1}$ be any value. Thanks to the  previous proposition, up to a change of $\mu$ by $\mu+t\delta$ for a very large $t$, we can assume that the combinatorial types of the solutions to the $\Delta$-problem with the value $\mu$ are surviving combinatorial types for our deformation. However, as noticed, the subtlety is that not all surviving combinatorial types provide a true solution. Nevertheless, each of the curves $\Gamma_i'$ is solution to the $\Delta_i$-problem with value $\mu|_{\widetilde{\Delta_i}}$. We use this to construct the solutions.\\

Let $\Gamma$ be a solution to the $\Delta$-problem with value $\mu+t\delta$ for $t$ large enough. By assumption it has a surviving combinatorial type. Moreover, the $\Gamma_i'$ provide solutions to the underlyings $\Delta_i$-problems with values $\mu|_{\widetilde{\Delta_i}}$, and the multiplicity of $\Gamma$ factors in the following way :
$$m_\Gamma^q=\prod_1^p [\sigma_i]_q m_{\Gamma_i}^q.$$

Conversely, any combinatorial type can be described by the ordered partition $\Delta-\{v_1,v_m\}=\bigsqcup\Delta_i$ along with the combinatorial type of the curves $\Gamma_i$. Let $(\Gamma_i)$ be a family of solutions to the respective $\Delta_i$-problems with respective values $\mu|_{\widetilde{\Delta_i}}$, and we try to glue them into a global solution, for $t$ large enough. The gluing is given by the order of the partition in which we glue $\Gamma_1,\dots,\Gamma_p$ on the chord linking the unbounded end $1$ to the unbounded end $m$. We have a formal solution obtained by resolving the length of the edges on the combinatorial type, and need to check that they indeed provide a true solution.\\

The lengths of the bounded edges inside the graphs $\Gamma_i$ are constant. The only orders that have non-decreasing lengths for the edges of the chords and the edges $E_i$ when $t$ goes to $+\infty$ are the orders considered in the formula. If the length of some of these edges is negative but increases through the deformation, it becomes positive for $t$ large enough. Conversely, if the order is not one of the considered, some edge length decreases along the deformation process and is ultimately negative, so the combinatorial type ceases to provide a solution. Finally, if the length of some edge of the chord is constant through the deformation, meaning that consecutive incidents vectors to the chord are colinear, there is a unique order between them that matches the order of their moments, \textit{i.e.} the order in which a transversal oriented line would meet them, and provides positive lengths for these edges, hence the consideration of the order on $\Gamma_1,\dots,\Gamma_p$ up to a reordering of consecutive colinear vectors $u_i$.\\

Finally, by putting together the contribution of the different possibilities of $\Gamma_i$, we get
$$\sum_{\Gamma_i\in\mathcal{S}(\mu|_{\widetilde{\Delta_i}})} \prod_1^p [\sigma_i]_q m_{\Gamma_i}^q = \prod_1^p [\sigma_i]_q N_{\Delta_i}^{\partial,\text{trop}},$$
and thus the desired formula.
\end{proof}

Provided that the deformation is big enough, the moments of all the unbounded edges except $v_1$ and $v_m$ are really small regarding these two specific moments. It means that when we look at a solution to the $\Delta$-problem with our value of $\mu$, all the edges adjacent to the chord seem to go through the origin of $\RR^2$ (Although they do not, but they are not far from it if we look at them from far far away) while the chord goes around the origin, changing its direction when meeting an adjacent edge in the right order.\\

This decomposition can be compared with the usual floor decomposition of tropical curves with an $h$-transverse degree coming from the tropical Caporaso-Harris formula of \cite{gathmann2007caporaso}. However, here the floors have a more complicated shape. For instance, even for degree $d$ curves, and the two edges whose moments vary are directed by $(-1,0)$ and $(0,-1)$, the chord may not be a usual floor and can make a loop around the origin as we can see on Figure \ref{quartic deformation}. The figure shows a quartic curve, with two unbounded ends going to infinity. The movement of these ends is depicted with an arrow. The region colored in grey is the zone through which the curve travels through the deformation. We have a similar situation for a cubic depicted on Figure \ref{cubic deformation}.


\begin{figure}
\begin{center}

\begin{tikzpicture}[line cap=round,line join=round,>=triangle 45,x=0.3cm,y=0.3cm]
\clip(0.,0.) rectangle (21.,19.);
\fill[line width=0.0pt,dash pattern=on 3pt off 3pt,color=black,fill=lightgray,fill opacity=0.1] (18.,0.) -- (17.99230503900855,10.99230503900855) -- (10.,11.) -- (4.,8.) -- (2.,6.) -- (2.,4.) -- (5.,1.) -- (7.,1.) -- (9.,3.) -- (14.,13.) -- (14.,17.) -- (0.,17.) -- (0.,16.) -- (13.,16.) -- (13.,12.) -- (9.,4.) -- (7.,2.) -- (5.,2.) -- (3.,4.) -- (3.,6.) -- (5.,8.) -- (9.,10.) -- (17.,10.) -- (17.,0.) -- cycle;
\draw (0.,4.)-- (3.,4.);
\draw (0.,6.)-- (3.,6.);
\draw (5.,0.)-- (5.,2.);
\draw (7.,2.)-- (7.,0.);
\draw (9.,4.)-- (9.,0.);
\draw (0.,8.)-- (5.,8.);
\draw (9.,10.)-- (16.,17.);
\draw [line width=2.pt] (9.,10.)-- (5.,8.);
\draw [line width=2.pt] (5.,8.)-- (3.,6.);
\draw [line width=2.pt] (3.,6.)-- (3.,4.);
\draw [line width=2.pt] (3.,4.)-- (5.,2.);
\draw [line width=2.pt] (5.,2.)-- (7.,2.);
\draw [line width=2.pt] (7.,2.)-- (9.,4.);
\draw [line width=2.pt] (9.,10.)-- (17.,10.);
\draw [line width=2.pt] (17.,10.)-- (17.,0.);
\draw (17.,10.)-- (20.,13.);
\draw [line width=2.pt] (9.,4.)-- (13.,12.);
\draw (13.,12.)-- (17.,16.);
\draw [line width=2.pt] (13.,12.)-- (13.,16.);
\draw (13.,16.)-- (15.,18.);
\draw [line width=2.pt] (13.,16.)-- (0.,16.);
\draw [line width=2.pt,dash pattern=on 3pt off 3pt] (18.,0.)-- (17.99230503900855,10.99230503900855);
\draw [line width=2.pt,dash pattern=on 3pt off 3pt,color=ttqqqq] (17.99230503900855,10.99230503900855)-- (10.,11.);
\draw [line width=2.pt,dash pattern=on 3pt off 3pt,color=ttqqqq] (10.,11.)-- (4.,8.);
\draw [line width=2.pt,dash pattern=on 3pt off 3pt,color=ttqqqq] (4.,8.)-- (2.,6.);
\draw [line width=2.pt,dash pattern=on 3pt off 3pt,color=ttqqqq] (2.,6.)-- (2.,4.);
\draw [line width=2.pt,dash pattern=on 3pt off 3pt,color=ttqqqq] (2.,4.)-- (5.,1.);
\draw [line width=2.pt,dash pattern=on 3pt off 3pt,color=ttqqqq] (5.,1.)-- (7.,1.);
\draw [line width=2.pt,dash pattern=on 3pt off 3pt,color=ttqqqq] (7.,1.)-- (9.,3.);
\draw [line width=2.pt,dash pattern=on 3pt off 3pt,color=ttqqqq] (9.,3.)-- (14.,13.);
\draw [line width=2.pt,dash pattern=on 3pt off 3pt,color=ttqqqq] (14.,13.)-- (14.,17.);
\draw [line width=2.pt,dash pattern=on 3pt off 3pt,color=ttqqqq] (14.,17.)-- (0.,17.);
\draw [line width=0.4pt,dash pattern=on 3pt off 3pt,color=ttqqqq] (18.,0.)-- (17.99230503900855,10.99230503900855);
\draw [line width=0.4pt,dash pattern=on 3pt off 3pt,color=ttqqqq] (17.99230503900855,10.99230503900855)-- (10.,11.);
\draw [line width=0.4pt,dash pattern=on 3pt off 3pt,color=ttqqqq] (10.,11.)-- (4.,8.);
\draw [line width=0.4pt,dash pattern=on 3pt off 3pt,color=ttqqqq] (4.,8.)-- (2.,6.);
\draw [line width=0.4pt,dash pattern=on 3pt off 3pt,color=ttqqqq] (2.,6.)-- (2.,4.);
\draw [line width=0.4pt,dash pattern=on 3pt off 3pt,color=ttqqqq] (2.,4.)-- (5.,1.);
\draw [line width=0.4pt,dash pattern=on 3pt off 3pt,color=ttqqqq] (5.,1.)-- (7.,1.);
\draw [line width=0.4pt,dash pattern=on 3pt off 3pt,color=ttqqqq] (7.,1.)-- (9.,3.);
\draw [line width=0.4pt,dash pattern=on 3pt off 3pt,color=ttqqqq] (9.,3.)-- (14.,13.);
\draw [line width=0.4pt,dash pattern=on 3pt off 3pt,color=ttqqqq] (14.,13.)-- (14.,17.);
\draw [line width=0.4pt,dash pattern=on 3pt off 3pt,color=ttqqqq] (14.,17.)-- (0.,17.);
\draw [line width=0.4pt,dash pattern=on 3pt off 3pt,color=ttqqqq] (0.,17.)-- (0.,16.);
\draw [line width=0.4pt,dash pattern=on 3pt off 3pt,color=ttqqqq] (0.,16.)-- (13.,16.);
\draw [line width=0.4pt,dash pattern=on 3pt off 3pt,color=ttqqqq] (13.,16.)-- (13.,12.);
\draw [line width=0.4pt,dash pattern=on 3pt off 3pt,color=ttqqqq] (13.,12.)-- (9.,4.);
\draw [line width=0.4pt,dash pattern=on 3pt off 3pt,color=ttqqqq] (9.,4.)-- (7.,2.);
\draw [line width=0.4pt,dash pattern=on 3pt off 3pt,color=ttqqqq] (7.,2.)-- (5.,2.);
\draw [line width=0.4pt,dash pattern=on 3pt off 3pt,color=ttqqqq] (5.,2.)-- (3.,4.);
\draw [line width=0.4pt,dash pattern=on 3pt off 3pt,color=ttqqqq] (3.,4.)-- (3.,6.);
\draw [line width=0.4pt,dash pattern=on 3pt off 3pt,color=ttqqqq] (3.,6.)-- (5.,8.);
\draw [line width=0.4pt,dash pattern=on 3pt off 3pt,color=ttqqqq] (5.,8.)-- (9.,10.);
\draw [line width=0.4pt,dash pattern=on 3pt off 3pt,color=ttqqqq] (9.,10.)-- (17.,10.);
\draw [line width=0.4pt,dash pattern=on 3pt off 3pt,color=ttqqqq] (17.,10.)-- (17.,0.);
\draw [line width=0.4pt,dash pattern=on 3pt off 3pt,color=ttqqqq] (17.,0.)-- (18.,0.);
\draw [line width=2.8pt] (16.,1.)-- (19.,1.);
\draw [line width=2.8pt] (18.063094110719252,1.) -- (17.781547055359628,0.6809133372590904);
\draw [line width=2.8pt] (18.063094110719252,1.) -- (17.781547055359628,1.3190866627409075);
\draw [line width=2.8pt] (17.5,1.) -- (17.218452944640372,0.6809133372590904);
\draw [line width=2.8pt] (17.5,1.) -- (17.218452944640372,1.3190866627409075);
\draw [line width=2.8pt] (1.,15.)-- (1.,18.);
\draw [line width=2.8pt] (1.,17.06309411071925) -- (1.3190866627409106,16.781547055359624);
\draw [line width=2.8pt] (1.,17.06309411071925) -- (0.6809133372590898,16.781547055359624);
\draw [line width=2.8pt] (1.,16.5) -- (1.3190866627409106,16.218452944640376);
\draw [line width=2.8pt] (1.,16.5) -- (0.6809133372590898,16.218452944640376);
\end{tikzpicture}

\caption{Image of a quartic curve during the deformation.\label{quartic deformation}}

\end{center}
\end{figure}
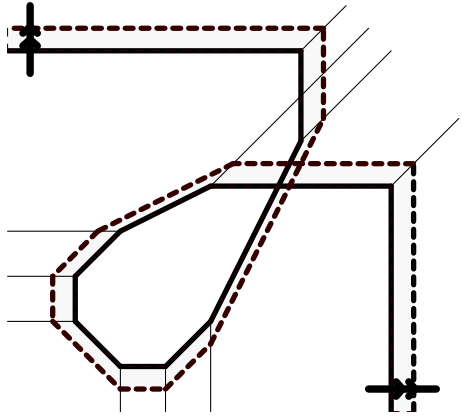


\begin{figure}
\begin{center}

\begin{tikzpicture}[line cap=round,line join=round,>=triangle 45,x=0.4cm,y=0.4cm]
\clip(-3.,-3.) rectangle (16.,16.);
\fill[line width=0.pt,color=lightgray,fill=lightgray,fill opacity=0.1] (0.,-2.) -- (1.,-2.) -- (1.,7.) -- (3.,9.) -- (7.,11.) -- (11.,11.) -- (11.,7.) -- (9.,3.) -- (7.,1.) -- (-2.,1.) -- (-2.,0.) -- (7.,0.) -- (9.,2.) -- (12.,8.) -- (12.,12.) -- (8.,12.) -- (2.,9.) -- (0.,7.) -- cycle;
\draw [line width=2.pt] (1.,7.)-- (3.,9.);
\draw [line width=2.pt] (7.,1.)-- (9.,3.);
\draw [line width=2.pt] (9.,3.)-- (11.,7.);
\draw [line width=2.pt] (3.,9.)-- (7.,11.);
\draw [line width=2.pt] (7.,11.)-- (11.,11.);
\draw [line width=2.pt] (11.,11.)-- (11.,7.);
\draw (7.,11.)-- (11.,15.);
\draw (11.,11.)-- (13.,13.);
\draw (11.,7.)-- (15.,11.);
\draw (3.,9.)-- (-2.,9.);
\draw (1.,7.)-- (-2.,7.);
\draw [line width=2.pt] (1.,7.)-- (1.,-2.);
\draw [line width=2.pt] (7.,1.)-- (-2.,1.);
\draw (7.,1.)-- (7.,-2.);
\draw (9.,3.)-- (9.,-2.);
\draw [line width=2.pt,dash pattern=on 3pt off 3pt] (-2.,0.)-- (7.,0.);
\draw [line width=2.pt,dash pattern=on 3pt off 3pt] (7.,0.)-- (9.,2.);
\draw [line width=2.pt,dash pattern=on 3pt off 3pt] (9.,2.)-- (12.,8.);
\draw [line width=2.pt,dash pattern=on 3pt off 3pt] (12.,8.)-- (12.,12.);
\draw [line width=2.pt,dash pattern=on 3pt off 3pt] (12.,12.)-- (8.,12.);
\draw [line width=2.pt,dash pattern=on 3pt off 3pt] (8.,12.)-- (2.,9.);
\draw [line width=2.pt,dash pattern=on 3pt off 3pt] (2.,9.)-- (0.,7.);
\draw [line width=2.pt,dash pattern=on 3pt off 3pt] (0.,7.)-- (0.,-2.);
\draw [line width=2.8pt] (-1.,2.)-- (-1.,-1.);
\draw [line width=2.8pt] (-1.,0.043256618374201465) -- (-1.2664336392817153,0.2716283091871003);
\draw [line width=2.8pt] (-1.,0.043256618374201465) -- (-0.7335663607182852,0.2716283091871003);
\draw [line width=2.8pt] (-1.,0.5) -- (-1.2664336392817153,0.7283716908128979);
\draw [line width=2.8pt] (-1.,0.5) -- (-0.7335663607182852,0.7283716908128979);
\draw [line width=2.8pt] (2.,-1.)-- (-1.,-1.);
\draw [line width=2.8pt] (0.043256618374202874,-1.) -- (0.2716283091871015,-0.7335663607182857);
\draw [line width=2.8pt] (0.043256618374202874,-1.) -- (0.2716283091871015,-1.2664336392817164);
\draw [line width=2.8pt] (0.5,-1.) -- (0.7283716908128988,-0.7335663607182857);
\draw [line width=2.8pt] (0.5,-1.) -- (0.7283716908128988,-1.2664336392817164);
\end{tikzpicture}

\caption{Image of a cubic curve during the deformation.\label{cubic deformation}}
\end{center}
\end{figure}
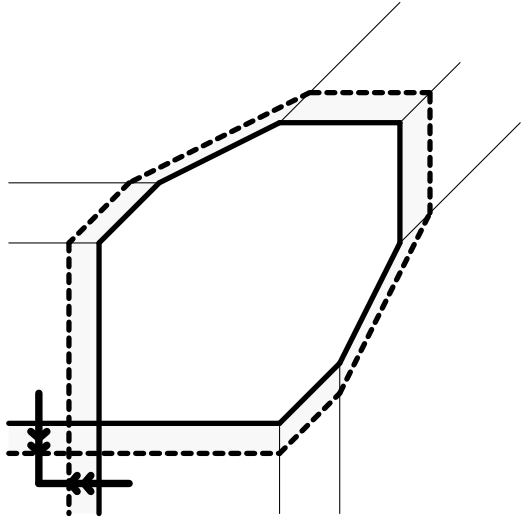


\section{Computations}

We now provide a few values $N_{\Delta}^{\partial,\text{trop}}$ for various families $\Delta\subset N$, computed using the recursive formula. Because of the exponential complexity of the algorithm, we only manage to make computations for small degrees. Concerning curves of degree $d$ in $\CC P^2$, computations can be done by hand up to degree $5$ or $6$. Degree $7$ seems to be out of reach without computer assistance.\\

Let $d\geqslant 1$ be an integer, and let $\lambda\vdash d$ be a partition of $d$. We denote by $N_d^{\partial,\text{trop}}(\lambda)$ the polynomial $N_\Delta^{\partial,\text{trop}}$ when $\Delta=\{ (-e_1)^d,(e_1+e_2)^d,-\lambda_1 e_2,-\lambda_2e_2,\dots\}$. These are the degrees that appear in the proof of the Caporaso-Harris formula in \cite{gathmann2007caporaso}.\\

\begin{prop}
We have :
\begin{itemize}[label=-]
\item $N_1^\partrop(1)=N_2^\partrop(1^2)=1$,
\item $N_2^\partrop(2)=q^{1/2}+q^{-1/2}$,
\item $N_3^\partrop(1^3)=q+7+q^{-1}$,
\item[ ]$N_3^\partrop(2,1)=q^{3/2}+6q^{1/2}+6q^{-1/2}+q^{-3/2}$,
\item[ ]$N_3^\partrop(3)=q^{2}+5q+6+5q^{-1}+q^{-2}$,
\item $N_4^\partrop(1^4)=q^3+10q^2+55q+172+55q^{-1}+10q^{-2}+q^{-3}$,
\item[ ]$N_4^\partrop(2,1^2)=q^{7/2}+9q^{5/2}+45q^{3/2}+133q^{1/2}+133q^{-1/2}+45q^{-3/2}+9q^{-5/2}+q^{-7/2}$,
\item[ ]$N_4^\partrop(3,1)=q^4+8q^3+36q^2+96q+117+96q^{-1}+36q^{-2}+8q^{-3}+q^{-4}$,
\item[ ]$N_4^\partrop(4)=q^{9/2}+7q^{7/2}+28q^{5/2}+68q^{3/2}+88q^{1/2}+88q^{-1/2}+68q^{-3/2}+28q^{-5/2}+8q^{-7/2}+q^{-9/2}$,
\item[ ]$N_4^\partrop(2^2)=q^4+8q^3+36q^2+104q+150+104q^{-1}+36q^{-2}+8q^{-3}+q^{-4}$,
\item $N_5^\partrop(1^5)=q^6+13q^5+91q^4+455q^3+1695q^2+5023q+11185+5023q^{-1}+1695q^{-2}+455q^{-3}+91q^{-4}+13q^{-5}+q^{-6}$.
\end{itemize}
\end{prop}

The proof is a straightforward computation. For each degree $\Delta$ one chooses two specific vectors and makes the associated unbounded edges going to infinity, reducing the computation of $N_\Delta^\partrop$ to the computations of invariants with families of smaller size. We here show some of the computations for degree $d$ curves, choosing unbounded ends directed by $(-1,0)$ and $(1,1)$. We explain only the main features, and draw the tropical curves resulting from the deformation. The shape of the tropical curves illustrate some of the involved phenomena.

\paragraph{Very low degrees.}
The values of $N_1^\partrop(1),N_2^\partrop(1^2)$ and $N_2^\partrop(2)$ are easy to find since for each choice of boundary conditions, there is only one tropical curve matching the constraints. The only curve for $N_2(2)$ has a vertex of complex multiplicity $2$, leading to the value $N_2^\partrop(2)=q^{1/2}+q^{-1/2}$.\\

For curves of degree $3$, the choice of unbounded edges going to infinity leads to tropical curves having a floor decomposition in the sense of \cite{gathmann2007caporaso}, and the computation is reduced to the value of $N_1^\partrop$, $N_2^\partrop(1^2)$ and $N_2^\partrop(2)$, which we already know. The appearance of a floor means that the chord degenerates into a toric divisor, which is the coordinate axis $y=0$ of $\CC P^2$ in our case. The computation leads to the value $q+7+q^{-1}$.

\paragraph{Curves of degree $4$ and $5$.}
For curves of degree $4$, we still get a contribution of the curves having a floor in the sense of \cite{gathmann2007caporaso}. Their contribution is
$$N_4^\partrop(1^4)=q^3+10q^2+55q+172+55q^{-1}+10q^{-2}+q^{-3}.$$
However, it would be wrong to assume that all the contributing curves are of this form. There are in fact $6$ additional curves having the shape of Figure \ref{loop quartic}, where the chord makes a loop around the origin. This means that the chord degenerates to the union of all toric divisors, rather than going to only one, as it would happen if it degenerated on a floor. The six curves come from the six possible repartitions of the bottom unbounded ends in two groups of two.

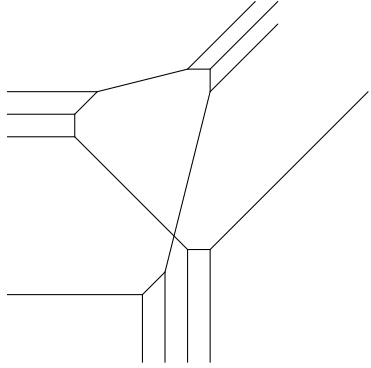
\begin{figure}
\begin{center}
\begin{tikzpicture}[line cap=round,line join=round,>=triangle 45,x=0.3cm,y=0.3cm]
\clip(-3.,0.) rectangle (13.,16.);
\draw (-3.,3.)-- (3.,3.);
\draw (3.,3.)-- (4.,4.);
\draw (4.,4.)-- (6.,12.);
\draw (6.,12.)-- (6.,13.);
\draw (6.,13.)-- (5.,13.);
\draw (5.,13.)-- (1.,12.);
\draw (1.,12.)-- (0.,11.);
\draw (0.,11.)-- (0.,10.);
\draw (0.,10.)-- (5.,5.);
\draw (5.,5.)-- (6.,5.);
\draw (6.,5.)-- (13.,12.);
\draw (0.,11.)-- (-3.,11.);
\draw (1.,12.)-- (-3.,12.);
\draw (0.,10.)-- (-3.,10.);
\draw (5.,13.)-- (8.,16.);
\draw (6.,13.)-- (9.,16.);
\draw (6.,12.)-- (9.,15.);
\draw (3.,3.)-- (3.,0.);
\draw (4.,4.)-- (4.,0.);
\draw (5.,5.)-- (5.,0.);
\draw (6.,5.)-- (6.,0.);
\end{tikzpicture}
\caption{Quartic curve with a loop}
\label{loop quartic}
\end{center}
\end{figure}

For curves of degree $5$, the chord can yet again degenerate into a floor, or just as in the degree $4$, degenerate into a loop. Only this time, contrarily to the degree $4$ case, there might still be unbounded ends leftover, leading to a line attached to the loop, as we can see on Figure \ref{loops quintics}.

\begin{figure}
\begin{center}
\begin{tabular}{cc}
\begin{tikzpicture}[line cap=round,line join=round,>=triangle 45,x=0.25cm,y=0.25cm]
\clip(-3.,0.) rectangle (13.,16.);
\draw (2.,3.)-- (3.,3.);
\draw (3.,3.)-- (4.,4.);
\draw (4.,4.)-- (6.,12.);
\draw (6.,12.)-- (6.,13.);
\draw (6.,13.)-- (5.,13.);
\draw (5.,13.)-- (1.,12.);
\draw (1.,12.)-- (0.,11.);
\draw (0.,11.)-- (0.,10.);
\draw (0.,10.)-- (5.,5.);
\draw (5.,5.)-- (6.,5.);
\draw (6.,5.)-- (13.,12.);
\draw (0.,11.)-- (-3.,11.);
\draw (1.,12.)-- (-3.,12.);
\draw (0.,10.)-- (-3.,10.);
\draw (5.,13.)-- (8.,16.);
\draw (6.,13.)-- (9.,16.);
\draw (6.,12.)-- (9.,15.);
\draw (3.,3.)-- (3.,0.);
\draw (4.,4.)-- (4.,0.);
\draw (5.,5.)-- (5.,0.);
\draw (6.,5.)-- (6.,0.);
\draw (2.,3.)-- (2.,10.5);
\draw (2.,10.5)-- (0.5,10.5);
\draw (2.,10.5)-- (3.,11.5);
\draw (1.5,2.5)-- (2.,3.);
\draw (1.5,2.5)-- (-3.,2.5);
\draw (1.5,2.5)-- (1.5,0.);
\end{tikzpicture}
&
\begin{tikzpicture}[line cap=round,line join=round,>=triangle 45,x=0.25cm,y=0.25cm]
\clip(-3.,0.) rectangle (13.,16.);
\draw (-3.,3.)-- (3.,3.);
\draw (3.,3.)-- (4.,4.);
\draw (4.,4.)-- (6.,12.);
\draw (6.,12.)-- (6.,13.);
\draw (6.,13.)-- (5.,13.);
\draw (1.,12.)-- (0.,11.);
\draw (0.,11.)-- (0.,10.);
\draw (0.,10.)-- (5.,5.);
\draw (5.,5.)-- (6.,5.);
\draw (6.,5.)-- (13.,12.);
\draw (0.,11.)-- (-3.,11.);
\draw (1.,12.)-- (-3.,12.);
\draw (0.,10.)-- (-3.,10.);
\draw (5.,13.)-- (8.,16.);
\draw (6.,13.)-- (9.,16.);
\draw (6.,12.)-- (9.,15.);
\draw (3.,3.)-- (3.,0.);
\draw (4.,4.)-- (4.,0.);
\draw (5.,5.)-- (5.,0.);
\draw (6.,5.)-- (6.,0.);
\draw (2.,10.5)-- (1.,10.5);
\draw (1.9969224263361691,10.5)-- (2.,9.);
\draw (1.9969224263361691,10.5)-- (4.3,12.8);
\draw (4.3,12.8)-- (5.,13.);
\draw (4.3,12.8)-- (3.8,12.8);
\draw (1.,12.)-- (3.2,12.8);
\draw (3.2,12.8)-- (3.8,12.8);
\draw (3.2,12.8)-- (7.,16.5);
\end{tikzpicture} \\
\end{tabular}
\caption{Quintics degenerating to a loop and a line}
\label{loops quintics}
\end{center}
\end{figure}
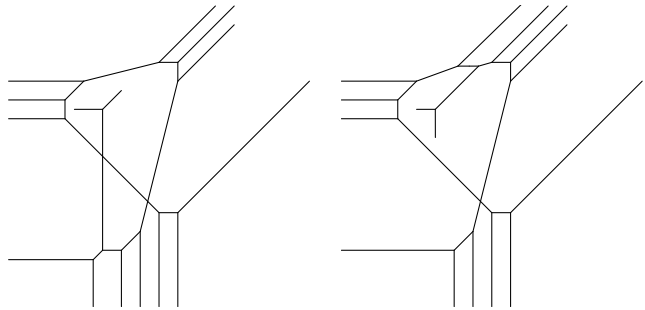

\paragraph{Degree $6$ and higher.}
Up to degree $5$, thanks to the particular choice of unbounded edges going to infinity, the computations were reduced only to values of the form $N_d^\partrop(\lambda)$ for some $\lambda$, as in the classical Caporaso-Harris formula. Once again, it would be wrong to assume that these values are sufficient to compute $N_d^\partrop(\lambda)$, in the sense that some smart choice of formula reduces the computation of $N_d^\partrop(\lambda)$ to the computation of some $N_l^\partrop(\mu)$ for $l<d$. Starting from $d=6$, the curves resulting from the eviction of the chord (\textit{i.e.} the curves $\Gamma_i$) may not be of degree $l$, meaning that the degree of the plane curve is not a standard triangle of size $l$, as we can see on Figure \ref{loop sextics}: the remaining curves might be of degree $1$ or $2$ as it is the case in $(c)$ and $(d)$, but can also have a more complicated shape, as we can see on $(a)$ and $(b)$.\\

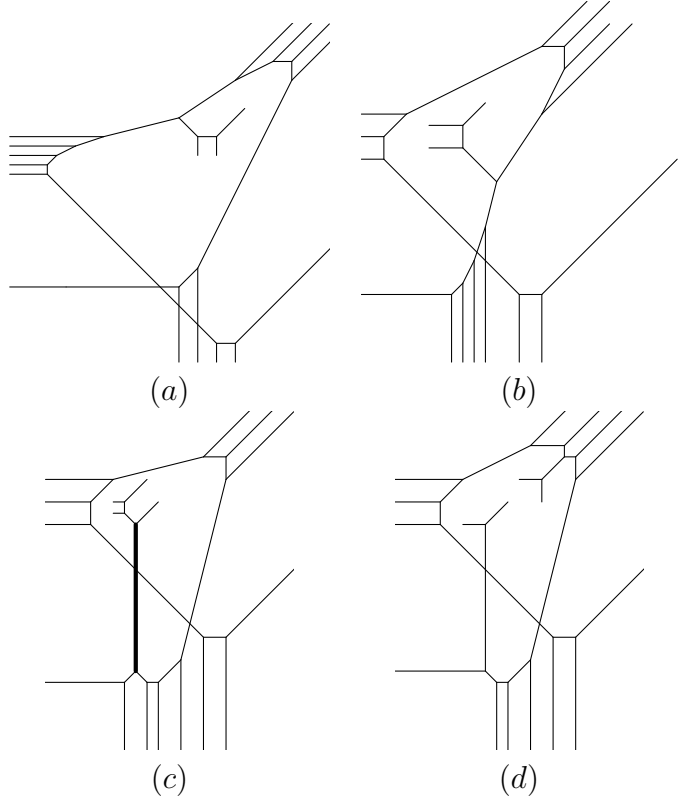
\begin{figure}
\begin{center}
\begin{tabular}{cc}
\begin{tikzpicture}[line cap=round,line join=round,>=triangle 45,x=0.25cm,y=0.25cm]
\clip(-3.,0.) rectangle (14.,18.);
\draw (0.,4.)-- (6.,4.);
\draw (6.,4.)-- (7.,5.);
\draw (7.,5.)-- (12.,15.);
\draw (12.,15.)-- (12.,16.);
\draw (12.,16.)-- (11.,16.);
\draw (11.,16.)-- (9.,15.);
\draw (9.,15.)-- (6.,13.);
\draw (6.,13.)-- (2.,12.);
\draw (2.,12.)-- (0.5,11.5);
\draw (0.5,11.5)-- (-0.5,11.);
\draw (-0.5,11.)-- (-1.,10.5);
\draw (-1.,10.5)-- (-1.,10.);
\draw (-1.,10.)-- (8.,1.);
\draw (8.,1.)-- (9.,1.);
\draw (9.,1.)-- (21.,13.);
\draw (6.,4.)-- (6.,0.);
\draw (7.,5.)-- (7.,0.);
\draw (8.,1.)-- (8.,0.);
\draw (9.,1.)-- (9.,0.);
\draw (12.,15.)-- (15.,18.);
\draw (12.,16.)-- (14.,18.);
\draw (11.,16.)-- (14.,19.);
\draw (9.,15.)-- (13.,19.);
\draw (0.,4.)-- (-3.,4.);
\draw (-1.,10.)-- (-3.,10.);
\draw (-1.,10.5)-- (-3.,10.5);
\draw (-0.5,11.)-- (-3.,11.);
\draw (0.5,11.5)-- (-3.,11.5);
\draw (2.,12.)-- (-3.,12.);
\draw (6.,13.)-- (7.,12.);
\draw (7.,12.)-- (7.,11.);
\draw (7.,12.)-- (8.,12.);
\draw (8.,12.)-- (8.,11.);
\draw (8.,12.)-- (9.5,13.5);
\end{tikzpicture}
&
\begin{tikzpicture}[line cap=round,line join=round,>=triangle 45,x=0.3cm,y=0.3cm]

\clip(0.,0.) rectangle (14.,16.);
\draw (0.,3.)-- (4.,3.);
\draw (4.,3.)-- (4.5,3.5);
\draw (4.5,3.5)-- (5.,4.5);
\draw (5.,4.5)-- (5.5,6.);
\draw (5.5,6.)-- (6.,8.);
\draw (6.,8.)-- (8.,11.);
\draw (8.,11.)-- (9.,13.);
\draw (9.,13.)-- (9.,14.);
\draw (9.,14.)-- (8.,14.);
\draw (8.,14.)-- (2.,11.);
\draw (2.,11.)-- (1.,10.);
\draw (1.,10.)-- (1.,9.);
\draw (1.,9.)-- (7.,3.);
\draw (7.,3.)-- (8.,3.);
\draw (8.,3.)-- (14.,9.);
\draw (4.,3.)-- (4.,0.);
\draw (4.5,3.5)-- (4.5,0.);
\draw (5.,4.5)-- (5.,0.);
\draw (7.,3.)-- (7.,0.);
\draw (8.,3.)-- (8.,0.);
\draw (5.5,6.)-- (5.5,0.);
\draw (1.,9.)-- (0.,9.);
\draw (1.,10.)-- (0.,10.);
\draw (2.,11.)-- (0.,11.);
\draw (8.,11.)-- (12.,15.);
\draw (9.,13.)-- (11.,15.);
\draw (9.,14.)-- (11.,16.);
\draw (8.,14.)-- (10.,16.);
\draw (6.,8.)-- (4.5,9.5);
\draw (4.5,9.5)-- (3.,9.5);
\draw (4.5,9.5)-- (4.5,10.5);
\draw (4.5,10.5)-- (3.,10.5);
\draw (4.5,10.5)-- (5.5,11.5);
\end{tikzpicture}
\\
$(a)$ & $(b)$ \\
\begin{tikzpicture}[line cap=round,line join=round,>=triangle 45,x=0.3cm,y=0.3cm]
\clip(-2.,0.) rectangle (9.,15.);
\draw (3.,3.)-- (4.,4.);
\draw (4.,4.)-- (6.,12.);
\draw (6.,12.)-- (6.,13.);
\draw (6.,13.)-- (5.,13.);
\draw (5.,13.)-- (1.,12.);
\draw (-2.,12.)-- (1.,12.);
\draw (1.,12.)-- (0.,11.);
\draw (0.,11.)-- (0.,10.);
\draw (0.,10.)-- (5.,5.);
\draw (5.,5.)-- (6.,5.);
\draw (6.,5.)-- (13.,12.);
\draw (0.,11.)-- (-3.,11.);
\draw (0.,10.)-- (-3.,10.);
\draw (5.,13.)-- (8.,16.);
\draw (6.,13.)-- (9.,16.);
\draw (6.,12.)-- (9.,15.);
\draw (3.,3.)-- (3.,0.);
\draw (4.,4.)-- (4.,0.);
\draw (5.,5.)-- (5.,0.);
\draw (6.,5.)-- (6.,0.);
\draw (3.,3.)-- (2.5,3.);
\draw (2.5,3.)-- (2.,3.5);
\draw (2.,3.5)-- (1.5,3.);
\draw (1.5,3.)-- (-2.5,3.);
\draw (1.5,3.)-- (1.5,0.);
\draw (2.5,3.)-- (2.5,0.);
\draw [line width=1.6pt] (2.,3.5)-- (2.,10.);
\draw (2.,10.)-- (3.,11.);
\draw (2.,10.)-- (1.5,10.5);
\draw (1.5,10.5)-- (1.5,11.);
\draw (1.5,11.)-- (2.5,12.);
\draw (1.5,11.)-- (1.,11.);
\draw (1.5,10.5)-- (1.,10.5);
\end{tikzpicture}
&
\begin{tikzpicture}[line cap=round,line join=round,>=triangle 45,x=0.3cm,y=0.3cm]
\clip(-2.,0.) rectangle (9.,15.);
\draw (3.,3.)-- (4.,4.);
\draw (4.,4.)-- (6.,12.);
\draw (6.,12.)-- (6.,13.);
\draw (1.,12.)-- (0.,11.);
\draw (0.,11.)-- (0.,10.);
\draw (0.,10.)-- (5.,5.);
\draw (5.,5.)-- (6.,5.);
\draw (6.,5.)-- (13.,12.);
\draw (0.,11.)-- (-3.,11.);
\draw (0.,10.)-- (-3.,10.);
\draw (6.,13.)-- (9.,16.);
\draw (6.,12.)-- (9.,15.);
\draw (3.,3.)-- (3.,0.);
\draw (4.,4.)-- (4.,0.);
\draw (5.,5.)-- (5.,0.);
\draw (6.,5.)-- (6.,0.);
\draw (3.,3.)-- (2.5,3.);
\draw (2.5,3.)-- (2.,3.5);
\draw (2.5,3.)-- (2.5,0.);
\draw (2.,3.5)-- (2.,10.);
\draw (2.,10.)-- (3.,11.);
\draw (2.,10.)-- (1.,10.);
\draw (2.,3.5)-- (-3.5,3.5);
\draw (4.5,12.)-- (4.5,11.);
\draw (4.5,12.)-- (3.5,12.);
\draw (6.,13.)-- (5.5,13.);
\draw (5.5,13.)-- (4.5,12.);
\draw (5.5,13.)-- (5.5,13.5);
\draw (5.5,13.5)-- (8.,16.);
\draw (5.5,13.5)-- (4.,13.5);
\draw (4.,13.5)-- (7.5,17.);
\draw (4.,13.5)-- (1.,12.);
\draw (1.,12.)-- (-3.,12.);
\end{tikzpicture}
\\
$(c)$ & $(d)$ \\
\end{tabular}
\caption{Different examples of sextics with a loop : in $(a)$ and $(b)$ the degree of the remaining curve is not a standard triangle, in $(c)$ it is a degree $2$ curve, in $(d)$ it is two lines.}
\label{loop sextics}
\end{center}
\end{figure}

For degree $7$, the situation becomes even more complicated since the growing number of unbounded edges increases the number of possible degrees for the curves $\Gamma_i$, and the chord can now make two loops. The number of loops that the chord can makes goes higher with the degree.\\

Finally, the recursive formula involves every $N_\Delta^\partrop$, associated with different toric surfaces, contrarily to the usual Caporaso-Harris formula for curves in $\CC P^2$, which restricts to specific degrees. Furthermore, the formula applies for curves of any degree, while the Caporaso-Harris formula restricts to $h$-transverse polygons.

\bibliographystyle{plain}
\bibliography{biblio}

{\ncsc Institut de Math\'ematiques de Jussieu - Paris Rive Gauche\\[-15.5pt] 

Sorbonne Universit\'e\\[-15.5pt] 

4 place Jussieu,
75252 Paris Cedex 5,
France} \\[-15.5pt]

\medskip

{\it and}

\medskip

{\ncsc D\'epartement de math\'ematiques et applications\\[-15.5pt]

Ecole Normale Sup\'erieure\\[-15.5pt]

45 rue d'Ulm, 75230 Paris Cedex 5, France} \\[-15.5pt]

{\it E-mail address}: {\ntt thomas.blomme@imj-prg.fr}

\end{document}